\documentclass[10pt]{article}
\usepackage[utf8]{inputenc}
\date{}

\usepackage{latexsym, color, amssymb, amsfonts,amsmath}
\usepackage{graphicx}
\usepackage{mathrsfs}
\usepackage{latexsym}
\usepackage{enumerate}
\usepackage{verbatimbox}

\usepackage{latexsym, color, amssymb, amsfonts,amsmath}
\usepackage{graphicx}
\usepackage{mathrsfs}
\usepackage{latexsym}
\usepackage{enumerate}
\usepackage{verbatimbox}

\usepackage{amsmath}
\usepackage{amsfonts}
\usepackage{amssymb}
\usepackage{color}
\usepackage{tikz}
\usepackage{mathtools}
\usepackage{bussproofs}
\usepackage{float}
\usetikzlibrary{arrows,backgrounds,positioning,fit,shapes,shapes.misc}

\newcommand{\red}[1]{{\color{red} #1}}

\newcommand{\negr}[1]{\boldsymbol{#1}}

\usepackage{euscript}

\usepackage{mathrsfs}
\usepackage{multicol}
\usepackage{euscript}

\usepackage{amssymb}
\usepackage{mathrsfs}
\usepackage{multicol}

\newenvironment{dem}{\noindent\bf Proof.\rm}{\hfill $\negr{\Box}$}

\newtheorem{proposition}{\bf Proposition}[section]

\newtheorem{theorem}[proposition]{\bf Theorem}
\newtheorem{lemma}[proposition]{\bf Lemma}
\newtheorem{facts}[proposition]{\bf Facts}
\newtheorem{corollary}[proposition]{\bf Corollary}
\newtheorem{remark}[proposition]{\bf Remark}
\newtheorem{remarks}[proposition]{\bf Remarks}
\newtheorem{example}[proposition]{\bf Example}
\newtheorem{definition}[proposition]{\bf Definition}
\newtheorem{defi}[proposition]{\bf Definition}
\newtheorem{prop}[proposition]{\bf Proposition}
\newtheorem{lem}[proposition]{\bf Lemma}
\newtheorem{theo}[proposition]{\bf Theorem}

\usepackage{latexsym, color, amssymb, amsfonts,amsmath}
\usepackage{graphicx}
\usepackage{mathrsfs}
\usepackage{latexsym}
\usepackage{enumerate}
\usepackage{verbatimbox}

\usepackage{amsmath}
\usepackage{amsfonts}
\usepackage{amssymb}
\usepackage{color}
\usepackage{tikz}
\usetikzlibrary{arrows,backgrounds,positioning,fit,shapes,shapes.misc}

\usepackage{mathrsfs}
\usepackage{multicol}

\def\iml{\!\rightarrowtail\!}
\def\imn{\!\rightarrow_{(n-1)}\!}
\def\imk{\!\rightarrow_{k}\!}
\def\ssi{\!\leftrightarrow\!}

\newenvironment{proof}{\noindent\bf Proof.\rm}{\hfill $\negr{\Box}$} 

\def\fun{\longrightarrow}

\def\max{\mathop{\sf Max}\nolimits}
\def\min{\mathop{\sf Min}\nolimits}
\newcommand{\free}[2]{\mathop{{\sf  Free}_{#1}}(#2)}

\def\A{\mathbf{A}}
\def\B{\mathbf{B}}
\def\S{\mathbf{S}}
\def\P{\mathcal{P}}
\def\F{\mathcal{F}}

\def\B{\mathbf{B}}

\def\C{\mathcal{C}}
\def\D{\mathcal{D}}
\def\E{\mathcal{E}}
\def\red{\color{red}}

\newcommand\restr[2]{{
  \left.\kern-\nulldelimiterspace 
  #1 
  \vphantom{\big|} 
  \right|_{#2} 
  }}
  
\def\Lukinf{{\cal \L$_{\infty}$ }}
\def\lukres{\mbox{\bf{\L}R}}
\def\luklog{{\cal {\L}H}^{\Delta}}

\newcounter{defcounter}
\newenvironment{myequation}{%
\addtocounter{equation}{-1}
\refstepcounter{defcounter}

\begin{equation}}
{\end{equation}}

\title{Super-\L ukasiewicz logics expanded by $\Delta$}

\author{A. V. Figallo$^a$,  A. Figallo-Orellano\footnote{E-mail: \texttt{aldofigallo@gmail.com}}\, and  M. Figallo\footnote{E-mail:\texttt{figallomartin@gmail.com}}  \\ [2mm] %
{\small $^a$ Instituto de Ciencias B\'asicas, Universidad Nacional de San Juan (UNSJ),} \\ {\small San Juan, Argentina,}\\
{\small $^\ast$Departamento de Informática e Matemática Aplicada, Universidade Federal do Rio Grande do Norte (UFRN),} \\ {\small Natal, Brazil,}\\
{\small $^\dag$Departamento de Matem\'atica and Instituto de Matemática (INMABB), Universidad Nacional del Sur (UNS),}\\
{\small  Bahia Blanca, Argentina}}

\begin{document}



 \maketitle

\begin{abstract}
Baaz's  operator $\Delta$ was introduced (by Baaz) in order to extend G\"odel logics, after that this operator was used to expand fuzzy logics by H\'ajek in his celebrated book. These logics were called $\Delta$-fuzzy logics. On the other hand, possibility operators were studied in the setting of \L ukasiewicz-Moisil algebras; curiously, one of these operators coincide with the Baaz's one. In this paper, we study the $\Delta$ operator in the context of  ($n$-valued) Super-\L ukasiewicz logics. An algebraic study of these logics is presented and the cardinality of Lindembaun-Tarski algebra with a finite number of variables is given. Finally, as a by-product, we present an alternative axiomatization of H\'ajek's \L ukasiwicz logic expanded with $\Delta$.
\end{abstract}

\


\section{Introduction and preliminaries}

As it is well-know, J. {\L}ukasiewicz introduced, in the early twenties of the last century, a class of propositional calculi by means of matrices  which has either some finite set of rationals within the real unit interval, or the whole unit interval as the truth values set, being the value $1$ the only designated truth value.
With  \Lukinf we shall denote the infinite--valued \L ukasiewicz propositional logic.
Recall that \Lukinf can be presented as the logic defined over the absolute free algebra of type $(2,1)$ with the axioms:
$$\phi \iml (\psi \iml \phi)$$
$$(\phi \iml \psi)\iml ((\psi \iml \eta)\iml (\phi \iml \eta))$$
$$((\phi\iml \psi)\iml \psi) \iml ((\psi\iml \phi)\iml \phi)$$
$$(\neg\phi\rightarrow \neg\psi )\iml (\psi \iml \phi)$$
and with {\em modus ponens} as a primitive rule.

\

B. Bosbach (\cite{Bos1, Bos2}) undertook the investigation of a class of residuated structures that were related to but considerably more general than Brouwerian semilattices and the algebras associated with the $\{\to, \wedge\}$-fragment of \L ukasiewicz's many--valued logics. \\
In a manuscript by J. B\"uchi and T. Owens (\cite{BuOw}) devoted to the study of Bosbach's algebras, written in the mid-seventies, the commutative members of this equational class were given the name {\em hoops}. An important subclass of the variety of hoops is the variety of {\em Wajsberg hoops}, so named and studied by W. Blok and I. Ferreirim (\cite{BlFe}). These algebras constitute the $\{\cdot, \to, 1\}$-subreducts of Wajsberg algebras. An important class of algebras studied by J. Berman and W. Blok in \cite{BB} is that of $\{\to,1\}$-subreducts of Wajsberg hoops which were called by these same authors {\em \L ukasiewicz residuation algebras} .\\
Recall that a  \L ukasiewicz residuation algebra is an algebra $\langle A,  \rightarrowtail, 1\rangle$ of type $(2,0)$ (for short, \lukres-algebra) that satisfies the following identities:

\begin{equation}\tag{\rm \L1} x\rightarrowtail(y\rightarrowtail x)\approx 1
 \end{equation}   
 \begin{equation}\tag{\rm \L2} (x\rightarrowtail y)\rightarrowtail((y\rightarrowtail z)\rightarrowtail(x\rightarrowtail z))\approx 1
 \end{equation}
 \begin{equation}\tag{\rm \L3} (x\rightarrowtail y)\rightarrowtail y \approx  (y\rightarrowtail x)\rightarrowtail x
 \end{equation}
  \begin{equation}\tag{\rm \L4}((x\rightarrowtail y)\rightarrowtail(y\rightarrowtail x))\rightarrowtail(y\rightarrowtail x)\approx 1
 \end{equation}
   \begin{equation}\tag{\rm \L5}1\rightarrowtail x\approx x
 \end{equation}
 
It is well-known that it is possible to define an order over every \L ukasiewicz residuation algebra $\bf A$ as follows: $x\leq y$ iff $x \rightarrowtail y=1$. Moreover, it is possible to define a supremum for every $x,y\in A$ through $x\vee y:= (x \rightarrowtail y) \rightarrowtail y$; and, also, we have that $z\leq 1$ for every $z\in A$. So, the axiom \L 4 can be expressed by $(x\rightarrowtail y) \vee (y\rightarrowtail x)=1$. 

Independently, Rodriguez Salas in 1980 (\cite{RS}) introduced Wajsberg algebras \break (or {\bf W}-algebras) as algebras $\A=\langle A, \iml, \neg, 1\rangle$ of type $(2, 1, 0)$ which satisfy  (\L5), (\L2) and (\L3) with the additional identity
$$(\neg y\iml \neg x)\iml (x\iml y) \approx 1$$
It was proved that Wajsberg algebras constitute a natural algebraization of the calculus \Lukinf. Besides, this same author defined and addition operation $+$ on $A$ as follows: for every $x,y\in A$
$$x+y=\neg y \iml x$$ 
and a product operation between a natural number $n$ and an element $x\in A$ as: \, $0\cdot a=0$, where $0\in A$ is $\neg 1$ and it is the first element of $\A$; and \, $(n+1)\cdot a = n\cdot a + a$. Then, he defined the {\em $n$-valued Wajsberg algebras} (or {\bf W}$_{n}$-algebras), $n\geq2$, as Wajsberg algebras satisfying the identity
$$\neg x \vee (n-1)\cdot x \approx 1$$

\

\noindent It is well-known that, on every \L ukasiewicz residuation algebra (or Wajsberg algebra), it is possible to define the binary operator  $\rightarrow_n$ as follows: $x \rightarrow_{0} y:=y$ \, and $x\rightarrow_{n+1} y:= x \rightarrowtail (x \rightarrow_{n} y)$, $n>0$. Then, a \L ukasiewicz residuation algebra $\langle A,  \rightarrowtail, 1\rangle$ is said to be $n$-valued (or \lukres$_n$-algebra), for $n\geq 2$, if it verifies:
\begin{equation}\tag{\rm \L6} (x\rightarrow_{(n-1)} y) \vee x\approx 1.
 \end{equation}
Note that, if $\A$ is a $\lukres_{n}$-algebra then, for every $x,y\in A$ we have that 
$$1=(x\imn y) \vee x \leq (x\rightarrow_{n}y)\vee x$$
and so, $\A$ is also a $\lukres_{n+1}$-algebra. In general, every $\lukres_n$-algebra $\A$ also is a $\lukres_{(n+k)}$-algebra for $k\geq0$.

\noindent At this point, we think it is important to recall some facts.

\begin{facts}\label{exam1}\begin{enumerate}[\rm (i)]
\item[]
\item We call \lukres$_0$-algebra to any algebra $\langle A, \iml, 1, 0\rangle$ of type $(2,0,0)$ such that the reduct \break $\langle A, \iml, 1\rangle \in \mbox{\lukres}$ and that it satisfies the identity
$$0\iml x\approx 1$$
If $\A=\langle A, \iml, 1, 0\rangle \in \lukres_0$, by ${\bf W}(\A)$ we denote the Wajsberg algebra $\langle A, \iml, \neg, 1\rangle$ where $\neg x:=x\iml 0$, for every $x\in A$. If $\A=\langle A, \iml, \neg, 1\rangle \in {\bf W}$ and we define $0:=\neg 1$, then \lukres$_0(\A)$ denotes the algebra $\langle A, \iml, 1, 0\rangle\in \mbox{\lukres}_0$; and \lukres$(\A)$ denotes the \lukres-algebra $\langle A, \iml, 1\rangle\in \lukres$.

\item If we consider over the unit real interval $[0,1]$ the \L ukasiewicz implication and negation defined as $x\iml y := \min\{1,1-x+y\}$ and $\neg x := 1 - x$, for $x,y\in[0,1]$, respectively, then $\langle [0,1], \iml, 1\rangle \in \mbox{\lukres}$,  $\langle [0,1], \iml, 1,0\rangle \in \mbox{\lukres}_0$  and $\langle [0,1], \iml, \neg, 1\rangle \in {\bf W}$.

\item  For each $n\geq 2$, let $\L_n=\{0, \frac{1}{n-1}, \dots, \frac{n-2}{n-1}, 1\}$. Then, {\bf \L}$_{n}=\langle \L_n, \iml, 1\rangle\in \mbox{\lukres}$ is the well-known \L ukasiewicz $n$-element chain, {\bf \L}$^{0}_{n}=\langle \{0, \frac{1}{n-1}, \dots, \frac{n-2}{n-1}, 1\}, \iml, 1, 0\rangle\in \mbox{\lukres}_0$ and {\bf \L}$^{\bf W}_{n}=\langle \{0, \frac{1}{n-1}, \dots, \frac{n-2}{n-1}, 1\}, \iml, \neg, 1\rangle\in {\bf W}$ where $\iml$ and $\neg$ are defined as in {\rm (ii)}. It is clear that ${\bf W}(\mbox{{\bf \L}}^{0}_{n})=\mbox{{\bf \L}}^{\bf W}_{n}$, $\mbox{\lukres}_0(\mbox{{\bf \L}}^{\bf W}_{n})=\mbox{{\bf \L}}^{0}_{n}$ and $\mbox{\lukres}(\mbox{{\bf \L}}^{\bf W}_{n})=\mbox{{\bf \L}}_{n}$.
\end{enumerate}
\end{facts}

\

\begin{proposition}  In every \lukres-algebra, the following hold ($k\geq 0$). 
 \begin{multicols}{2}
\begin{enumerate}
  \item[\rm (\L7)] $x\rightarrowtail 1\approx 1,$

  \item[\rm (\L8)] $x\preceq y$ implies $y\rightarrowtail z\preceq x\rightarrowtail z,$ 
  \item[\rm  (\L9)] $x\rightarrowtail(y\rightarrowtail z)\approx y\rightarrowtail(x\rightarrowtail z),$
  \item[\rm  (\L10)] $x\rightarrowtail x\approx 1,$
  \item[\rm  (\L11)] $x\preceq y$ implies $z\imk x\preceq z\imk y$, 
   \item[\rm  (\L12)] $y\preceq x\rightarrowtail y,$
  \item[\rm  (\L13)] $(x\rightarrowtail y)\rightarrowtail(x\rightarrowtail z) \preceq x\rightarrowtail(y\rightarrowtail z),$
  \item[\rm  (\L14)] $(x\vee y)\rightarrowtail y\approx x\rightarrowtail y,$
  \item[\rm  (\L15)] $(x\rightarrowtail y)\rightarrowtail((z\rightarrowtail x)\rightarrowtail(z\rightarrowtail y))\approx 1,$
\item[\rm (\L 16)] $x\rightarrow_{k}(y\rightarrowtail z) = y \rightarrowtail (x\rightarrow_{k} z)$.
\end{enumerate}
\end{multicols}
\end{proposition}
\begin{dem}  For (\L11) and (\L 16), we use induction on $k$. The rest  are consequence of  {\rm (\L1)}--{\rm (\L5)}. \, (\L 11) : If $k=0$, the assertion holds trivially; and for $k=1$ the assertion is consequence of (\L 2). Let $x,y\in A$ such that $x\leq y$ and suppose that $z\imk x\leq z\imk y$, by the case $k=1$, $z\iml(z\imk x)\leq z\iml(z\imk y)$, i.e, $z\rightarrow_{k+1}x \leq z\rightarrow_{k+1}y$.\\
(\L 16): \, If $k=0$, the assertion holds trivially; and for $k=1$ the equation (\L16) holds by (\L9). Suppose that (I.H.) \, $x\rightarrow_{k}(y\rightarrowtail z) = y \rightarrowtail (x\rightarrow_{k} z)$. Then 
\begin{center}
\begin{tabular}{rcl}
$x\rightarrow_{k+1}(y \rightarrowtail z)$& $=$ & $x\rightarrowtail(x\rightarrow_{k}(y \rightarrowtail z))$ \\
 & $\stackrel{(I.H.)}{=}$ & $x\rightarrowtail(y \rightarrowtail(x\rightarrow_{k} z))$ \\
  & $\stackrel{(\L9)}{=}$ & $y\rightarrowtail(x \rightarrowtail(x\rightarrow_{k} z))$ \\
   & $=$ & $y\rightarrowtail(x \rightarrow_{k+1} z)$
\end{tabular}
\end{center}
\end{dem}

\

\begin{remark}\label{rem1}\rm
In every  \lukres$_n$-algebra $\A$ ($n\geq 2$)  we have that $(x\imn y)\vee x=1$, for any $x,y\in A$. That is, $(x\iml(x\imn y))\iml(x\imn y)=1$ and then $x\iml(x\imn y)\leq x\imn y$. On the other hand, by (\L 12) and (\L16), we have \, (i) \, $x\imn y = x\imn(x\iml y)$. From this, we can be proved \, (ii) \, $x\imn(x\imn y)=x\imn y$ \, and \, 
\begin{equation}\tag{\L17}x\imn(y\rightarrowtail z) = (x\imn y) \rightarrowtail (x\imn z)
\end{equation}
Indeed,  $1\stackrel{(\L2)}{=} (x\iml y)\iml((y\iml z)\iml (x\iml z))\stackrel{(\L9)}{=}(y\iml z)\iml((x\iml y)\iml (x\iml z))$ \break and so $(y\iml z)\leq (x\iml y)\iml (x\iml z)$. By (\L11), 
$$x\imn(y\iml z)\leq  x\imn((x\iml y)\iml (x\iml z)) \stackrel{(\L 16)}{=}(x\iml y)\iml(x\imn(x\iml z)) \stackrel{(i)}{=} (x\iml y)\iml(x\imn z).$$
Then 
\begin{center}
\begin{tabular}{ll}
& $(x\imn(y\iml z)\big)\iml ((x\iml y)\iml(x\imn z))=1$\\[1.5mm]
$\stackrel{(\L 9)}{\Longleftrightarrow}$ \hspace*{1cm}& $(x\iml y)\iml \big((x\imn(y\iml z))\iml(x\imn z)\big)=1$\\[1.5mm]
$\stackrel{}{\Longleftrightarrow}$ \hspace{2cm}& $(x\iml y)\leq \big(\big(x\imn(y\iml z)\big)\iml(x\imn z)\big)$\\[1.5mm]
$\stackrel{(\L 11)}{\Longleftrightarrow}$ \hspace{2cm}& $x\imn(x\iml y)\leq x\imn\big(\big(x\imn(y\iml z)\big)\iml(x\imn z)\big)$\\[1.5mm]
$\stackrel{(\L 12), (\L9)}{\Longleftrightarrow}$ \hspace{2cm}& $x\imn  y\leq x\iml(x\imn  y)\leq \big(x\imn(y\iml z)\big)\iml\big(x\imn(x\imn z)\big)$\\[1.5mm]
$\stackrel{(ii)}{\Longleftrightarrow}$ \hspace{2cm}& $x\imn  y\leq (x\imn(y\iml z))\iml\big(x\imn z\big)$\\[1.5mm]
$\stackrel{}{\Longleftrightarrow}$ \hspace{2cm}& $\big(x\imn  y\big)\iml \big((x\imn(y\iml z))\iml\big(x\imn z\big)\big)=1$\\[1.5mm]
$\stackrel{(\L 9)}{\Longleftrightarrow}$ \hspace{2cm}& $(x\imn(y\iml z))\iml\big((x\imn  y)\iml(x\imn z\big)\big)=1$\\[1.5mm]
$\stackrel{}{\Longleftrightarrow}$ \hspace{2cm}& $(x\imn(y\iml z))\leq(x\imn  y)\iml(x\imn z\big)$\\[1.5mm]
\end{tabular}
\end{center}
The other inequality is proved analogously.
\end{remark}

\

\noindent The following proposition shows a remarkable relation between $\imn$ and $\rightarrow_{k}$.

\begin{proposition}\label{prop22} Let $\A\in \mbox{\lukres}_n$, $n\geq 2$. Then, for every $x,y,z \in A$ it holds:
\begin{itemize}
\item[\rm (i)] $x\imn(y\rightarrow_{k} z) = (x\imn y) \rightarrow_{k} (x\imn z)$,
\item[\rm (ii)] $(x\imn y)\rightarrow_{k}x=x$ \, if \, $k\geq 1$.
\end{itemize}
 \end{proposition}
\begin{dem} (i): \, For $k=0$, the equality is immediate. (I.H.) Suppose that it holds for $k>0$. Then 
\begin{center}
\begin{tabular}{rcl}
$(x\imn y) \rightarrow_{k+1} (x\imn z)$& $=$ & $(x\imn y)\rightarrowtail ((x\imn y)\rightarrow_{k} (x\imn z))$ \\
 & $\stackrel{(I.H.)}{=}$ & $(x\imn y) \rightarrowtail (x\imn(y\rightarrow_{k} z))$ \\
  & $\stackrel{(\L17)}{=}$ & $x\imn(y\rightarrowtail (y\rightarrow_{k} z))$ \\
   & $=$ & $x\imn(y\rightarrow_{k+1} z)$
\end{tabular}
\end{center}

\

\noindent (ii): (B.S.) \,  If $k=1$, from (\L6), we know that $((x\imn y)\rightarrowtail x)\rightarrowtail x =1$ and therefore $((x\imn y)\rightarrowtail x)\leq x$. On the other hand, from (\L1) we have  $x\rightarrowtail((x\imn y)\rightarrowtail x)=1$ and then $x\leq(x\imn y)\rightarrowtail x$. \\
(I.H.) Suppose that $(x\imn y)\rightarrow_{k}x=x$ \, for \, $k\geq 1$. Then,
\begin{center}
\begin{tabular}{rcl}
$(x\imn y)\rightarrow_{k+1}x$& $=$ & $(x\imn y)\rightarrowtail((x\imn y)\rightarrow_{k}x)$ \\
 & $\stackrel{(\L17)}{=}$ & $(x\imn y)\rightarrow_{k}((x\imn y)\rightarrowtail x)$ \\
  & $\stackrel{B.S.}{=}$ & $(x\imn y)\rightarrow_{k}x$ \\
   & $\stackrel{(I.H.)}{=}$ & $x$
\end{tabular}
\end{center}
\end{dem}

\


\begin{proposition}(A.V. Figallo)\label{lemaux} Let $\A$ be an \lukres$_n$-algebra, $n\geq 2$. Then, for any $x,y,z\in A$ it holds:
\begin{itemize}
\item[\rm (\L18) ] $1\rightarrow_k x=x$, $k\geq 0$;
\item[\rm (\L19) ] $x\rightarrow_k 1=1$, $k\geq 0$;
\item[\rm (\L20) ] $x\rightarrow_k x=1$, $k\geq 1$;
\item[\rm (\L21) ] $x\imn (y\imn z)=(x\imn y)\imn (x\imn z)$; 
\item[\rm (\L22)] $x\imn (y\imn  x)=1$; 
\item[\rm (\L23)] $((x\imn y)\imn x)\imn x=1$.
\end{itemize}
\end{proposition}
\begin{dem} {\rm (\L18)}, {\rm (\L19)} and {\rm (\L20)} are easily obtained using induction on $k$. \\[2mm]
{\rm (\L21)} From Proposition \ref{prop22}, taking $n-1=k$. \\[2mm]
{\rm (\L22)} \,  $x\imn(y\imn  x) \, \stackrel{(\L21)}{=} \, (x\imn y)\imn(x\imn x) \, \stackrel{(\L20)}{=} \break\, (x\imn y)\imn 1 \, \stackrel{(\L19)}{=} \, 1$. \\[2mm]
{\rm (\L23)} From Proposition \ref{prop22}(ii) and (\L20) we have \, $((x\imn y)\imn x)\imn x  = \break x\imn x =1.$
\end{dem}

\

\begin{remarks}\begin{enumerate}[\rm (i)]
\item[]
\item  \lukres$_{2}$ is precisely the well-known class of {\em Tarski} algebras.

\item {\bf \L}$_{n}$ of Facts \ref{exam1} (iii) is a \lukres$_n$-algebra, for $n\geq 2$.

\item  Let $n\geq 2$ and let $k$ a natural number such that $2\leq k \leq n$, then {\bf \L}$_{k}$ is isomorphic to a \lukres$_{n}$--subalgebra of {\bf \L}$_{n}$. Indeed, let $S$ be  the upper set $\left\uparrow\frac{n-k}{n-1}\right.$, then it is clear that $S$ is a \lukres$_{n}$--subalgebra of {\bf \L}$_{n}$. Besides, $S$ is a chain with $k$ elements and,  by \cite[pg. 124]{RS}, we know that {\bf \L}$_{k}$ is \lukres$_n$-isomorphic to $S$. This marks an essential difference between the classes \lukres$_{n}$ and ${\bf W}_{n}$ since it is well-known that \,  {\bf \L}$^{\bf W}_{k}$ is a ${\bf W}_{n}$-subalgebra of \, {\bf \L}$^{\bf W}_{n}$ \, iff \, $(k-1)\mid (n-1)$ ($(k-1)$ divides $(n-1)$).
\end{enumerate}
\end{remarks}

\

\section{ $\Delta$-\L ukasiewicz residuation algebras of order $n$}\label{s3}

\noindent In 1990, A. V. Figallo introduced the $3$-valued Super-\L ukasiewicz logic expanded with $\Delta$  in \cite{AVF} motivated by  Moisil operators (see also \cite{AVF2}) and the fact that, in this logic, it is not possible to recover $\Delta$ from the implication and bottom (see \cite[Example 1]{AVF1}).

In the area of {\em fuzzy logic}, the operator $\Delta$ was studied and called  $\Delta$ operator  by Baaz, \cite{Baaz}. This author studied the  propositional and the quantified version of G\"odel logic expanded by $\Delta$. Later on, Hájek studied the extensions of  Basic Fuzzy Logic (BL), \L ukasiewicz logic, Product logic and other fuzzy logics by the $\Delta$ operator, \cite{PH}. In this setting, Esteva and Godo introduced the logic MTL and its extension MTL$_\Delta$ by $\Delta$ operator, \cite{EG-01}, see also \cite{EGM2001,EGHN2000}. Furthermore, Hájek and Cintula  called all these systems  $\Delta$-fuzzy logics and presented their quantified version  (without identity) with the respective soundness and completeness theorems in \cite{HC}.  Their completeness proof for these first-order logics is  obtained by adding the axiom of {\em constant domains} and using a similar Henkin's strategy.

Next, we consider an equational class of algebras which constitute a natural generalization of the $I\Delta_3$-algebras studied in \cite{AVF}.

\begin{definition}
A $\Delta$-\L ukasiewicz residuation algebra of order $n$, $n\geq 2$,  is an algebra $(A,\rightarrowtail,\Delta,1 )$ of type $(2, 1, 0)$ such that $(A, \rightarrowtail, 1)$ is an \lukres$_n$-algebra and the following identities are satisfied: 
\begin{equation}\tag{\rm $\Delta$\L1}
\Delta x\rightarrowtail y \approx x \imn y
\end{equation}
\begin{equation}\tag{\rm $\Delta$\L2} \Delta(\Delta x\rightarrowtail y)\approx \Delta x\rightarrowtail\Delta y
\end{equation}
\end{definition}

\

\noindent We denote the variety of $\Delta$-\L ukasiewicz residuation algebras of order $n$ by   \lukres$^{\Delta}_n$. If $\A \in \mbox{\lukres}^{\Delta}_n$ and $X\subseteq A$, then $\Delta(X)=\{\Delta x: x\in X\}$. We shall note $\Delta(\A$) instead of $\Delta(A)$. Besides, recall that $x\in A$ is said to be a {\em Tarskian element} if for every $y\in A$ it holds \, 
$$ x\rightarrowtail y = x\rightarrowtail (x\rightarrowtail y).$$ 
We denote with $T(\A)$ the set of all Tarskian elements of $\A$.

\begin{lemma}
Let  $\A$ be an \lukres$^{\Delta}_n$-algebra. For every $x,y, z\in A$ the following conditions are satisfied:
 \begin{multicols}{2}
\begin{itemize}
    \item[\rm ($\Delta$\L3)] $\Delta x\leq x$
    \item[\rm ($\Delta$\L4)] $x\imn\Delta x =1$ 
    \item[\rm ($\Delta$\L5)] $\Delta 1=1$
    \item[\rm ($\Delta$\L6)] $\Delta\Delta x = \Delta x$
    \item[\rm ($\Delta$\L7)] $\Delta(\Delta x\rightarrowtail \Delta y)=\Delta x\rightarrowtail \Delta y$
    \item[\rm ($\Delta$\L8)] $\Delta x\rightarrowtail y=\Delta x\rightarrowtail (\Delta x\rightarrowtail y)$
    \item[\rm ($\Delta$\L9)] $T(\A)=\Delta(\A)$
    \item[\rm ($\Delta$\L10)] $\Delta x\rightarrowtail (y\rightarrowtail  z)=(\Delta x\rightarrowtail y)\rightarrowtail(\Delta x\rightarrowtail z)$
    \item[\rm ($\Delta$\L11)] If $x\leq y$, then $\Delta x\leq \Delta y$
    \item[\rm ($\Delta$\L12) ] $\Delta(x\rightarrow_n y)=\Delta x\rightarrowtail \Delta y$    
    \item[\rm ($\Delta$\L13)] $\Delta x\rightarrowtail\Delta(x\rightarrowtail\Delta x)=1$
    \item[\rm ($\Delta$\L14)] $\Delta x\rightarrowtail\Delta(x\rightarrowtail y)=\Delta x\rightarrowtail\Delta y$
    \item[\rm ($\Delta$\L15)] $\Delta(x\rightarrowtail y)\leq \Delta x\rightarrowtail\Delta y$
\end{itemize}
\end{multicols}
\end{lemma}

\begin{proof} We just prove  ($\Delta$\L6),  ($\Delta$\L8), ($\Delta$\L9), ($\Delta$\L14) and ($\Delta$\L15), the remaining are analogous.\\
{\rm ($\Delta$\L6)}: From ($\Delta$\L3), we have $\Delta\Delta x \leq \Delta x$. On the other hand, \, $\Delta x \rightarrowtail \Delta\Delta x \, \stackrel{\rm (\Delta\L2)}{=} \break \Delta(\Delta x \rightarrowtail \Delta x) \, \stackrel{\rm (\L10)}{=} \, \Delta 1 \, \stackrel{\rm (\Delta\L5)}{=} \, 1$; and therefore, $\Delta x\leq\Delta\Delta x$. \\
{\rm ($\Delta$\L8)}: \,  $\Delta x\rightarrowtail (\Delta x\rightarrowtail y) \, \stackrel{\rm (\Delta\L1)}{=} \, x\imn(x\imn y) \, = \, x\imn y  \, \stackrel{\rm (\Delta\L1)}{=} \, \Delta x \rightarrowtail y$, by Proposition \ref{prop22}(ii).\\
{\rm ($\Delta$\L9)}: \, From  ($\Delta$\L8), we have that $\Delta(\A) \subseteq T(\A)$. Conversely, let $x\in T(\A)$, then \break $x\rightarrowtail \Delta x = x\rightarrowtail (x\rightarrowtail \Delta x)$. Then, by the definition of $\imn$, we have $ x\rightarrowtail \Delta x = x\imn \Delta x$. Then, $x \rightarrowtail \Delta x = x\imn \Delta x  \, \stackrel{\rm (\Delta\L4)}{=} \, 1$, therefore $x\leq \Delta x$ and by 
 ($\Delta$\L3),  we have that $\Delta x=x$.\\
 {\rm ($\Delta$\L14)}: 
 \begin{center}
\begin{tabular}{rcl}
$ \Delta x\rightarrowtail \Delta(x\rightarrowtail y)$&  $\stackrel{(\Delta\L2)}{=}$ & $\Delta(\Delta x\rightarrowtail (x\rightarrowtail y))$ \\
 & $=$ & $\Delta((\Delta x\rightarrowtail x)\rightarrowtail (\Delta x\rightarrowtail y))$ \\
   & $\stackrel{(\Delta\L3)}{=}$ & $\Delta(1\rightarrowtail (\Delta x\rightarrowtail y))$ \\
   & $=$ & $\Delta(\Delta x\rightarrowtail y)$ \\
     & $\stackrel{(\Delta\L2)}{=}$ & $\Delta x\rightarrowtail \Delta y$ \\
\end{tabular}
\end{center}
 {\rm ($\Delta$\L15)}: 
  \begin{center}
\begin{tabular}{rcl}
$ \Delta (x\rightarrowtail y)\rightarrowtail( \Delta x\rightarrowtail \Delta y)$&  $\stackrel{(\L9)}{=}$ & $\Delta x\rightarrowtail( \Delta (x\rightarrowtail y)\rightarrowtail \Delta y)$ \\
 & $=$ & $(\Delta x\rightarrowtail(\Delta x\rightarrowtail y))\rightarrowtail (\Delta x\rightarrowtail\Delta y)$ \\
   & $\stackrel{(\Delta\L14)}{=}$ & $(\Delta x\rightarrowtail\Delta y)\rightarrowtail (\Delta x\rightarrowtail\Delta y)$\\
   & $=$ & $1$ 
  \end{tabular}
\end{center}
\end{proof}

\

\begin{example}\label{exampleLn} For each $n\geq 2$, consider the operator $\Delta$ defined over  {\bf \L}$_{n}=\langle \L_n, \iml, 1\rangle\in \lukres_n$ by
 $$\Delta x = \left\{ \begin{tabular}{ll} $0$ & $x\not=1$ \\ $1$ & $x=1$ \end{tabular} \right.$$ 
for every $x\in \L_n$. Then, {\bf \L}$^{\Delta}_{n}=\langle \L_n, \iml, \Delta, 1\rangle$ is an \lukres$^{\Delta}_n$-algebra.
\end{example}

\

Let $\langle A, \rightarrowtail, 1\rangle$ be an $n$-valued \L ukasiewicz residuation algebra and assume that $\Delta_1$ and $\Delta_2$ are two operators defined on $A$ such that both satisfy {\rm $\Delta$\L1} and {\rm $\Delta$\L2}. Then, if $x\in A$ we know that \, 
$$\Delta_1 x \iml\Delta_2 x \stackrel{\rm (\Delta\L1)}{=} x \imn \Delta_2 x \stackrel{\rm (\Delta\L1)}{=} \Delta_2 x  \iml \Delta_2x = 1 $$
Hence, $\Delta_1 x\leq\Delta_2 x$. Analogously we show that  $\Delta_2 x\leq\Delta_1 x$. So, we have shown the following:

\begin{lemma} Let  $A \in \mbox{\lukres}_n$. Then, $A$ admits at most one  structure of \lukres$^{\Delta}_n$-algebra.
\end{lemma}

\

\noindent However, we shall see that not every $n$-valued \L ukasiewicz residuation algebra admits a structure of \lukres$^{\Delta}_n$-algebra. Indeed, let $\A$ be an \lukres$^{\Delta}_n$-algebra. For every $x\in A$ we define the set $T_x$ as the set of all Tarskian elements of $\A$ which are less than or equal to $x$, i.e.,
$$T_x=\{ t\in T(\A) : t\leq x\}$$
Then,
\begin{proposition}\label{prop29} If $\A$ is an \lukres$^{\Delta}_n$-algebra and $x\in A$. Then, \, $T_x$ has greatest element and it coincides with $\Delta x$.
\end{proposition}
\begin{dem} Let $y\in T_x$. From ($\Delta$\L8) and ($\Delta$\L3) we have that $\Delta x\in T_x$. Besides, if $t\in T_x\subseteq T(\A)$ then, by ($\Delta$\L9), $t\in \Delta(\A)$. That is, $t=\Delta y$, for some $y\in A$. Then, since $t\leq x$, by ($\Delta$\L6) and ($\Delta$\L11), $\Delta t=\Delta\Delta y=\Delta y=t\leq\Delta x$. 
\end{dem}

\

\noindent Now, consider the \lukres$_3$-algebra $\A=\langle A, \iml, 1\rangle$  where $A=\{a,b,c,d,1 \}$ and whose underlying ordered structure is given by the following Hasse diagram

\

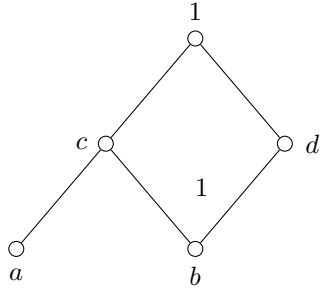
\hspace{1cm}
\begin{tikzpicture}[scale=.7]
\tikzstyle{every node}=[draw,circle,fill=white,inner sep=2pt]
  \node (one) at (0,3) [label=above:$1$] {};
  \node (c) at (-1.7,1) [label=left:$c$] {};
   \node (d) at (1.7,1) [label=right:$d$] {};
  \node (b) at (0,-1) [label=below:$ b$] {};
 \node (a) at (-3.4,-1) [label=below:$ a$] {};
  \draw (a) -- (c) -- (one) -- (d) -- (b) -- (c);
\hspace{5cm}
\begin{tabular}{c|ccccc}
$\iml$ & $a$ & $b$ & $c$ & $d$ & $1$ \\ \hline
$a$ & $1$ & $d$ & $1$ & $d$ & $1$ \\ 
$b$ & $c$ & $1$ & $1$ & $1$ & $1$ \\ 
$c$ & $c$ & $d$ & $1$ & $d$ & $1$ \\ 
$d$ & $a$ & $c$ & $c$ & $1$ & $1$ \\ 
$1$ & $a$ & $b$ & $c$ & $d$ & $1$ 
\end{tabular}
\end{tikzpicture}  

\

\

\noindent Then, $T(\A) =\{a, d, 1\}$ and $T_{b}=\emptyset$ and therefore $\A$ does not admit a structure of \break \lukres$^{\Delta}_3$-algebra. This shows that it is not always possible to define the operator $\Delta$ in a given $n$-valued \L ukasiewicz residuation algebra; and therefore, our study is well founded.

\

\begin{lemma}\label{lem210} If $\A\in \mbox{\lukres}^{\Delta}_n$ and $t\in T(\A)$, then $x\imn t \in T(\A)$ for every $x\in A$.
\end{lemma}
\begin{dem} Let $x,z\in A$ and $t\in T(\A)$, then
\begin{center}
\begin{tabular}{rcl}
& &$ ((x\imn t)\iml ((x\imn t)\iml z)) \iml ((x\imn t)\iml z)$\\[1mm]
&$\stackrel{(\L3)}{=}$ &$(((x\imn t)\iml z))\iml(x\imn t))\iml (x\imn t)$ \\[1mm]
 & $\stackrel{(\L16)}{=}$ & $(x\imn(((x\imn t)\iml z)\iml t))\iml (x\imn t)$\\[1mm]
   & $\stackrel{(\L17)}{=}$ & $x\imn[(((x\imn t)\iml z)\iml t)\iml t]$\\[1mm]
   & $\stackrel{(\L3)}{=}$ & $x\imn[(t\iml((x\imn t)\iml z))\iml ((x\imn t)\iml z)]$\\[1mm]
& $\stackrel{(\L17)}{=}$ & $x\imn[((t\iml(x\imn t))\iml (t\iml z))\iml ((x\imn t)\iml z)]$\\[1mm]
  \end{tabular}
\end{center}
\

\noindent But $t\iml(x\imn t)\stackrel{(\L16)}{=} x\imn (t\iml t) \stackrel{(\L10)}{=} x\imn 1 \stackrel{(\L19)}{=}1$. Then,
\begin{center}
\begin{tabular}{rcl}
&  $=$ & $x\imn[(1\iml (t\iml z))\iml ((x\imn t)\iml z)]$\\[1mm]
 & $\stackrel{(\L5)}{=}$ & $x\imn[(t\iml z)\iml ((x\imn t)\iml z)]$ \\[1mm]
   & $\stackrel{(\L9)}{=}$ & $x\imn[(x\imn t)\iml ((t\iml z)\iml z)]$\\[1mm]
   & $\stackrel{(\L3)}{=}$ & $x\imn[(x\imn t)\iml ((z\iml t)\iml t)]$\\[1mm]
& $\stackrel{(\L16)}{=}$ & $(x\imn t)\iml[x\imn ((z\iml t)\iml t)]$\\[1mm]
& $\stackrel{(\L16)}{=}$ & $(x\imn t)\iml[(z\iml t)\iml(x\imn  t)]$\\[1mm]
& $\stackrel{(\L1)}{=}$ & $1$
  \end{tabular}
\end{center}
\noindent Hence, $(x\imn t)\iml ((x\imn t)\iml z) \leq (x\imn t)\iml z$ and,  by  (\L12), $(x\imn t)\iml z\leq (x\imn t)\iml ((x\imn t)\iml z)$. Therefore, $(x\imn t)\iml z =  (x\imn t)\iml ((x\imn t)\iml z)$ and $x\imn t\in T(\A)$.
\end{dem}

\

\noindent We end this section establishing necessary and sufficient conditions for an arbitrary $n$-valued \L ukasiewicz algebra to admit  a structure of $\mbox{\lukres}^{\Delta}_n$-algebra.

\begin{theorem} Let $\A=\langle A, \iml, 1\rangle$ be an arbitrary \lukres$_n$-algebra. Then, the following conditions are equivalent.
\begin{enumerate}[\rm (i)]
\item $\A$ admits a structure of $\mbox{\lukres}^{\Delta}_n$-algebra,
\item for every $x\in A$, $T_x$ has greatest element.
\end{enumerate}
\end{theorem}
\begin{dem}(i) $\Longrightarrow$ (ii): It is consequence of Proposition \ref{prop29}.\\[1.5mm]
(ii) $\Longrightarrow$ (i): Let $s:A\fun A$ the operation defined by 
$$sx=\max\{t\in T(\A) : t\leq x\}=\max T_x.$$
Then, it is not difficult to check that, for every $x,y\in A$, they hold: (1) \, $sx\leq x$, (2) \, $ssx=sx$ and (3) if $x\leq y$ then $sx\leq sy$. Besides, (4) if $z\in T(\A)$ and $z\leq x$, then $z\leq sx$ \, and \,  (5) if $y\in T(\A)$ then $y= sy$.
Finally, for every $x,y\in A$, since $sy\in T(\A)$ by Lemma \ref{lem210}, $sx\iml sy\in T(\A)$ and by (5), we have \, (6) $s(sx\iml sy)=sx \iml sy$.\\
Let us prove:\\
(7)\, \underline{$sx\iml sy \leq s(sx\iml y)$}. Indeed, by (1), $sy\leq y$ and by (\L11) \, $sx\iml sy\leq sx\iml y$. Then, from (3) and (6), $sx\iml sy\leq s(sx\iml y)$.\\[1.5mm]
(8)\,  \underline{$x \imn y \leq sx\iml y$}. By (1), $sx\leq x$ then, by (\L8), $x\iml y\leq sx\iml y$ and then (8) holds for $n=1$. (I.H.) Suppose that  $x \imn y \leq sx\iml y$. By (\L11), we have (*) $x\iml(x\imn y)\leq x\iml(sx\iml y)$. On the other hand,
\begin{center}
\begin{tabular}{rcl}
(**) \, $x\iml(sx\iml y)$&  $\stackrel{(\L9)}{=}$ & $sx\iml(x\imn y)$ \\
 & $\stackrel{(\L13)}{=}$ & $(sx\iml x)\iml (sx\iml y)$ \\
   & $\stackrel{(1)}{=}$ & $1\iml (sx\iml y)$ \\
   & $\stackrel{(\L5)}{=}$ & $sx\iml y$\\ 
  \end{tabular}
\end{center}
Then, $x\rightarrow_{n+1} y = x\iml(x\imn y) \, \stackrel{(*)}{\leq} \, x\iml(sx\iml y) \, \stackrel{(**)}{=} \, sx\iml y$.\\[1.5mm]
(9) \, \underline{$x\imn sx =1$}. From (\L12), $sx\leq (x\iml sx)\iml sx$. Since $sx\in T(\A)$ and by Lemma\ref{lem210}, $x\imn sx\in T(\A)$ \, and by (5), \, (*) $s(x\imn sx)=x\imn sx$. On the other hand, by Proposition \ref{prop22}, (**) \, $sx=s((x\imn sx)\iml x) \, \stackrel{(*)}{=} \, s(s(x\imn sx)\iml x)$. From (7), we have \,  $s(x\imn sx)\iml sx  \stackrel{(7)}{\leq} s(s(x\imn sx)\iml x)  \stackrel{(**)}{=}  sx$; and therefore  $(s(x\imn sx)\iml sx) \iml sx=1$. That is, 
$$1\, =  \, (s(x\imn sx)\iml sx)\iml sx  \, \stackrel{(*)}{=} \, ((x\imn sx)\iml sx)\iml sx  = (x\imn sx)\vee sx=x\imn sx$$
Since, $sx\iml(x\imn sx) \stackrel{(\L16)}{=} x\imn (sx\iml sx)\stackrel{(\L10)}{=} x\imn 1 \stackrel{(\L7)}{=} 1$, and then $sx\leq x\imn sx$.\\[1.5mm]
(10) \, \underline{$sx\iml y = x\imn y$}. From
\begin{center}
\begin{tabular}{rcl}
$(sx\iml y)\iml(x\imn y)$&  $\stackrel{(\L16)}{=}$ & $x\imn((sx\iml y)\iml y)$ \\
 & $\stackrel{(\L3)}{=}$ & $x\imn((y\iml sx)\iml sx)$ \\
   & $\stackrel{(\L16)}{=}$ & $(y\iml sx)\iml(x\imn sx)$ \\
   & $\stackrel{(9)}{=}$ & $(y\iml sx)\imn 1 = y\iml sx $\\ 
  \end{tabular}
\end{center}
we have $sx\iml y \leq x\imn y$; and by (8),  $sx\iml y = x\imn y$.\\[1.5mm]
(11) \, \underline{$s(sx\iml y) = sx\iml sy$}.
\begin{center}
\begin{tabular}{rcl}
$s(sx\iml y) \iml (sx\iml sy)$&  $\stackrel{(10)}{=}$ & $(sx\iml y) \imn (sx\iml sy)$ \\
 & $\stackrel{(2)}{=}$ & $(ssx\iml y) \imn (ssx\iml sy)$ \\
   & $\stackrel{(10)}{=}$ & $(sx\imn y) \imn (sx\imn sy)$ \\
   & $\stackrel{(\L17)}{=}$ & $sx\imn (y\imn sy) $\\ 
   & $\stackrel{(9)}{=}$ & $sx\imn 1$\\ 
   & $\stackrel{(\L19)}{=}$ & $1 $\\ 
  \end{tabular}
\end{center}
Hence, defining $\Delta x = sx$, we have that $\langle A, \iml, \Delta, 1\rangle$ is an $\lukres^{\Delta}_n$-algebra.
\end{dem}

\

It is worth mentioning that, in \cite{AVF0}, it was considered the class of $n$-valued \L ukasiewicz residuation algebras, $n\geq 2$, enriched with the well-known {\em Moisil possibility operators}. More presicely, let $J=\{1,\dots, n\}$, the structure $\langle A, \iml, (\Delta_i)_{i\in J}, 1\rangle$ of type $(2,(1)_{i\in J}, 0)$ is an  $n$-valued \L ukasiewicz residuation algebras with Moisil possibility operators if \, $\langle A, \iml, 1\rangle \in \lukres_{n}$ and the following identities hold:
\begin{equation}\tag{\rm M\L1}
\Delta_1 x\rightarrowtail y \approx x \rightarrow_n y
\end{equation}
\begin{equation}\tag{\rm M\L2}
\Delta_i x \vee \Delta_i x \iml y  \approx 1, \, i\in J
\end{equation}
\begin{equation}\tag{\rm M\L3}
\Delta_i (\Delta_j x \iml \Delta_j y)  \approx \Delta_j x \iml \Delta_j y , \, i, j\in J
\end{equation}
\begin{equation}\tag{\rm M\L4}
(\Delta_1 x \iml \Delta_1 y)\iml ((\Delta_2 x \iml \Delta_2 y)\iml\dots\iml((\Delta_n x \iml \Delta_n y)\iml(x\iml y))\dots)\approx 1
\end{equation}
\begin{equation}\tag{\rm M\L5}
\Delta_i y \iml(\Delta_j x \vee \Delta_k (x\iml y))\approx 1, \, \, \,  j,k\in J,  \, \, \,  1\leq i \leq j+k
\end{equation}
\begin{equation}\tag{\rm M\L5}
\Delta_i (x\iml y) \iml  (\Delta_k x \iml \Delta_j y) \approx 1, \, \, \,  j,k\in J,  \, \, \,  1\leq i \leq j-k+1
\end{equation}

\

\noindent Later, in \cite{AVF1}, it was described a method for constructing the operators $\Delta_i$, $2\leq i\leq n$, from $\Delta_1$ and $\iml$, in every \lukres$_{n}$-algebra.

\begin{remark}\label{remMPO} If $\langle A, \iml, (\Delta_i)_{i\in J}, 1\rangle$ is an  $n$-valued \L ukasiewicz residuation algebras with Moisil possibility operators, then $\langle A, \iml, \Delta_1, 1\rangle \in \lukres^{\Delta}_n$. Conversely, given $\langle A, \iml, \Delta, 1\rangle \in \lukres^{\Delta}_n$, then taking $\Delta_1 x :=\Delta x$, it is possible to construct operators $\Delta_i$ on $A$, $2\leq i\leq n$,  in such a way that $\langle A, \iml, (\Delta_i)_{i\in J}, 1\rangle$ turns out to be an $n$-valued \L ukasiewicz residuation algebras with Moisil possibility operators.
\end{remark}

 \subsection{Implicative filters and simple algebras}\label{imfiltersect}
 

Recall that if $\A$ is an \L ukasiewicz residuation algebra,  $D \subseteq A$ is said to be  an implicative filter of $A$  if (1) $1\in D$ and (2) $D$t is closed under modus ponens,  i.e.,  for every $x,y\in A$ it holds: if $x,x\rightarrowtail y\in D$, then $y\in D$. If $\A$ is a $\cal K$-algebra for ${\cal K}\in \{\lukres, \lukres_{n}, \lukres^{\Delta}_{n}, {\bf W}_{n}\}$, we denote by $\D_{\cal K}(\A)$ and  $\E_{\cal K}(\A)$ the set of all implicative filters and maximal implicative filters of $\A$, respectively. 

\

\begin{lemma} Let $\A$ be a \L ukasiewicz residuation algebra and $D\in\D_{\lukres}(A)$. Then, $D$ is closed under {\rm $k$-weak modus ponens}, that is, for every $x,y\in A$ and $k\geq 0$ we have
$$x,x\rightarrow_{k} y\in D  \mbox{ \, implies \, } y\in D.$$
\end{lemma}
\begin{dem} We use induction on $k$. If $k=0$ or $k=1$ the result is obvious. Suppose that, $D$ is closed under $k$-weak modus ponens, and assume that (1) $x \in D$ and $x\rightarrow_{k+1} y \in D$. That is,  $x\iml(x\rightarrow_{k} y) \in D$ and by definition of implicative filter and (1),  we have (2) $x\rightarrow_{k} y\in D$. Then, from (1), (2) and the inductive hipothesis, $y\in  D$.
\end{dem}

\

\noindent Then, it is clear that if $\A$ is  a \L ukasiewicz residuation algebra then \, 
\begin{quote}
$D\in\D(A)$ \, iff \, $1\in D$ and $D$ is closed under $k$-weak modus ponens. \hfill (IF)
\end{quote}

\noindent For any class of algebras $\cal K$ and every $\bf A \in {\cal K}$, we denote by $Con_{\cal K}(\A)$ the set of all $\cal K$-congruences of $\A$.\\[2mm]
Now, let $\A\in \lukres^{\Delta}_n$. For any implicative filter $D$ of $\bf A$, we define (as usual) the equivalence relation $R(D)$ on $A$ as
$$R(D)=\{(x,y)\in A^2:x\rightarrow_n y,y\rightarrow_n x\in D\}.$$ 
More over, it is easy to show that $R(D)$ is a congruence relation of $\A$. We call $R(D)$  the congruence relation of $\A$ associated to $D$. 

\begin{remark}
It is not difficult to check that, given a congruence $\Theta$ of $\A$ we have that $|1|_{\Theta}$, the equivalence class of $1$ by $\Theta$, is an implicative filter. Besides, $ R(|1|_{\Theta})=\Theta$. 
\end{remark}

\begin{proposition}\label{prop211} Let $\A\in \mbox{\lukres}^{\Delta}_n$ and $D\in \D_{\mbox{\lukres}^{\Delta}_n}(\A)$. Then, for every $x\in A$,  
$$x\in D \mbox{ \, iff \, } \Delta x\in D.$$
\end{proposition}
\begin{dem} Let $x\in D$. By ($\Delta$\L4), \, $x\imn\Delta x\in D$ and so $\Delta x\in D$. The converse is consequence of ($\Delta$\L3).
\end{dem}

\

\begin{lemma}\label{lem2.12} Let $\A\in \mbox{\lukres}^{\Delta}_n$. Then,  $Con_{\mbox{\lukres}_n}(\A)=Con_{\mbox{\lukres}^{\Delta}_n}(\A)$.
\end{lemma}
\begin{dem} We just have to show that $Con_{\mbox{\lukres}_n}(\A) \subseteq Con_{\mbox{\lukres}^{\Delta}_n}(\A)$. Let $R\in Con_{\mbox{\lukres}_n}(\A)$, then  $R=R(|1|_{R})$. Let $(x,y)\in R$, then (1)$x \iml y\in |1|_{R}$ and $y \iml x\in |1|_{R}$. From (1) and Proposition \ref{prop211}, \, $\Delta(x\iml y)\in |1|_{R}$ and by ($\Delta$\L15), \,  $\Delta x\iml \Delta y \in |1|_{R}$. Analogously we get that $\Delta y\iml \Delta x \in |1|_{R}$ and by (IF), \, $(\Delta x, \Delta y)\in  R$. Therefore, $R\in Con_{\mbox{\lukres}^{\Delta}_n}(\A)$.
\end{dem}

\


\


\begin{lemma}\label{lemprodsub} \, If $\A\in\mbox{\lukres}^{\Delta}_n$ is non trivial, then $\A$ is a subdirect product of simple $\mbox{\lukres}^{\Delta}_n$-algebras.
\end{lemma}
\begin{dem} It is consequence of Lemma \ref{lem2.12}.
\end{dem}

\begin{lemma}
There exists a lattice-isomorphism between $Con_{\mbox{\lukres}^{\Delta}_n}(\A)$ and $\D_{\mbox{\lukres}^{\Delta}_n}(\A)$.
\end{lemma}
\begin{dem} It is consequence of Lemma \ref{lem2.12}.
\end{dem}

\begin{definition}{\rm (A. Monteiro).}\label{Ligado}
Let $\bf A$ be an  \lukres$^{\Delta}_n$-algebra, $D\in\D_{\mbox{\lukres}^{\Delta}_n}(A)$ and $p\in A$. We say that $D$ is an  implicative filter tied  to $p$ if $p\notin D$ and for any $D'\in\D_{\mbox{\lukres}^{\Delta}_n}(A)$ such that $D\subsetneq D'$, then $p\in D'$.
\end{definition}

Recall that for a given \lukres$^{\Delta}_n$-algebra $\bf A$, we say that an implicative filter $M$ is maximal if $M$ is proper and for any $D\in \D_{\mbox{\lukres}^{\Delta}_n}(A)$,  $M\subseteq D$ implies $D=A$ or $M=D$. 

\begin{lemma}{\rm (\cite[Lemma 6.4]{FOS1}).}\label{Mont-2}
Let $\bf A$ be an \lukres$^{\Delta}_n$-algebra and $M$ a maximal implicative filter of $A$. Then, for every $x\in A\setminus M$, we have that $x\rightarrow_n y\in A$  for every $y\in A$.
\end{lemma}

For a given \lukres$^{\Delta}_n$-algebra $\bf A$ and according to ($\Delta$\L4), (\L21), (\L22) and (\L23), we can prove the following corollary and theorem, respectively, following \cite{FOS1}.

\begin{corollary}\label{maximal}
For a given \lukres$^{\Delta}_n$-algebra $\bf A$, each implicative filter  tied to some element of $A$ is maximal and vice versa.
\end{corollary}
 
\begin{theorem} \label{teoavf} 
The variety of  \lukres$^{\Delta}_n$-algebra is semisimple.  
\end{theorem}
\begin{proof}
Taking into account Lemma  \ref{lemaux}, we have that the class of  $\lukres^{\Delta}_n$-algebras is in fact a $M$-algebra of Section 6 of \cite{FOS1}, so from Lemma 6.5 of the same paper we have proved the Theorem.
\end{proof}

As it said above, in  \cite{GPS2020}, it was studied the $n$-valued \L ukasiewicz residuation algebras expanded with Moisil possibility operators. Taking into account Remark \ref{remMPO}, it is clear that each  $\lukres^{\Delta}_n$-algebras  can be seen as a $n$-valued \L ukasiewicz residuation algebras expanded with Moisil possibility operators. Then,


\begin{lemma} Let $\A\in \lukres^{\Delta}_n$ and let $(\Delta_i)_{i\in J}$, where $J=\{1,\dots, n\}$, the family of unary operators indicated in Remark \ref{remMPO}. Then, the following identities hold in $\A$.
\begin{itemize}
\item[{\rm (M\L 7)}]  $\Delta_j 1\approx 1,$ $j\in J$

\item[{\rm (M\L 8)}]  $\Delta_1 x\leq \Delta_2x\leq \dots \leq \Delta_{n-1} x,$

\item[{\rm (M\L 9)}]  $\Delta_j x\iml (\Delta_j x\iml y) \approx \Delta_j x\iml y$, $j\in J,$

\item[{\rm (M\L 10)}]  $\Delta_j x\iml y \approx \Delta_j x \rightarrow_n
y,$ $j\in J,$ 

\item[{\rm (M\L 11)}]  $(\Delta_j x\iml y)\to \Delta_j x \approx \Delta_j x$, $j\in J,$

\item[{\rm (M\L 12)}]  $(x\iml y)\iml (\Delta_j x \to \Delta_j y)\approx 1$,  $j\in J,$

\item[{\rm (M\L 13)}]   $x\leq y$ implies $\Delta_j x\leq
\Delta_j y$,  $j\in J,$

\item[{\rm (M\L 14)}]  $\Delta_1 x\leq x,$

\item[{\rm (M\L 15)}]  $\Delta_j x\leq \Delta_j y$ for all $j\in J$ implies $x\leq y,$

\item[{\rm (M\L 16)}]  $\Delta_k\Delta_j x\approx \Delta_j x,$ $k, j\in J,$

\item[{\rm (M\L 17)}]  $x\leq \Delta_{n-1} x$

\item[{\rm (M\L 18)}]  $x\iml \Delta_1x \approx 1.$
\end{itemize} 
\end{lemma}
\begin{proof}
It follows from (M\L 1) to (M\L 6) and taking into account the proof given in \cite[Lemma 2.4]{GPS2020}.
\end{proof}

\begin{theo}\label{teoMj}
	Let $\langle A, \iml, \Delta, 1\rangle$ be an \lukres$^{\Delta}_n$-algebra and  let  $M$ be a maximal deductive system of $A$. Now, consider the set $M_j$ for $0\leq j\leq n$ defined by
	$$M_j = \left\{\begin{array}{ll} \{x\notin M: \Delta_n x\not\in M\} & \mbox{ if } j=0 \\[1.5mm]
	   \{x\notin M: \Delta_{ n-i-j} x\not\in M\,\, \text{and}\,\, \Delta_{n-1-j} x \in M \}& \mbox{ if } 1\leq j<n \\[1.5mm]
	M & \mbox{ if } $j=n$ \end{array}\right.$$
Then, the function $h:A\to \L_n$ given by $h(x)=\frac{s}{n}$ if $x\in M_s$ with $0\leq s\leq n$ is a homomorphism such that $h^{-1}(\{1\})=M$.
\end{theo}
 
\begin{proof}
It follows from (M\L 7), (M\L 8), (M\L 16), and and taking into account the proof given in \cite[Theorem 2.11]{GPS2020}, that $h$ is an homomorphism which verifies that  $h^{-1}(\{1\})=M$. Now, from the first isomorphism theorem, we have there is a one-to-one homomorphism from $A/M$ into {\bf \L}$^{\Delta}_{n}$ as desired.
\end{proof}

\begin{corollary}\label{lemsim} For $n\geq 2$, the simple algebras of $\lukres^{\Delta}_n$ are precisely (up to isomorphism) {\bf \L}$^{\Delta}_{k}$ \,  for $k\leq n$.
\end{corollary}

\subsection{The free $\mbox{\lukres}^{\Delta}_n$-algebra with a finite number of generators}

In this subsection, we study the structure of the free  $\Delta$-\L ukasiewicz residuation algebra of order $n$, $n\geq2$. Given a class $\cal K$ of algebras, by $\free{\cal K}{G}$ we denote the free $\cal K$--algebra generated by the set $G$ of free generators. When $|G|=m<\omega$, we use the notation $\free{\cal K}{m}$. Besides, for every poset $P$ we denote by $\mu(P)$ the set of its minimal elements. Besides, given a set $X$, by $|X|$ we denote the cardinality (size) of $X$.

\

\begin{lemma}\label{lemfree1} \,  Let $G$ be a set of free generators of a \lukres$^{\Delta}_n$-algebra $\A$. Then, the posets $G$ and $\Delta G$ are antichains.
\end{lemma}
\begin{dem} Let us see first that $G$ is an antichain. If $|G|=1$ the lemma holds trivially. Suppose that $|G|>1$ and let $g_1,g_2\in G$ such that $g_1\not= g_2$. Assume that $g_1< g_2$ and consider the function  $f:G\fun \mbox{{\bf \L}}^{\Delta}_{n}$ defined by
$$f(t)=\left\{\begin{array}{ll} 1 & \mbox{ if } t=g_1  \\ $0$ & \mbox{ otherwise } \end{array} \right.$$ 
Then, there exists a unique \lukres$^{\Delta}_{n}$-homomorphism $h:\free{\mbox{\lukres}^{\Delta}_n}{G} \fun \mbox{{\bf \L}}^{\Delta}_{n}$ which extends $f$. Hence, $h(g_1)=f(g_1)=1>0=f(g_2)=h(g_2)
$ which is a contradiction since any   \lukres$^{\Delta}_{n}$-homomorphism is an order-preserving map. Similarly, we see that $g_2\not<g_1$ and therefore $g_1$ and $g_2$ are incomparable elements of $G$, and then $G$ is an antichain.\\
For $\Delta G$ the proof goes analogously. In particular, if $\Delta g_1 < \Delta g_2$ and considering the same \lukres$^{\Delta}_{n}$-homomorphism $h$ we have: $h(\Delta g_1) = \Delta h(g_1) = \Delta f(g_1)=\Delta 1\stackrel{(\Delta\L5)}{=} 1$; and,  $h(\Delta g_2) = \Delta h(g_2) = \Delta f(g_2)=\Delta 0 =0$, that is, $h(\Delta g_1)> h(\Delta g_2)$.
\end{dem}

\

\begin{corollary} \label{corfree}
\begin{enumerate}[\rm (i)] \item[]
\item  $\mu(G)= G$,
\item  $\mu(\Delta G)=\Delta G$,
\item  $|G|=|\Delta G|$.
\end{enumerate}
\end{corollary}

\

\begin{lemma}\label{free1} Let $G$ be a set of free generators. Then $\mu \left(\free{\mbox{\lukres}^{\Delta}_n}{G}\right)  = \Delta G$.
\end{lemma}
\begin{dem} Let $S=\{x \in \free{\mbox{\lukres}^{\Delta}_n}{G} :  \Delta g \leq x \mbox{ for some } g \in G\}\subseteq \free{\mbox{\lukres}^{\Delta}_n}{G}$.  Let us see that $S$ is a $\mbox{\lukres}^{\Delta}_n$-subalgebra of $\free{\mbox{\lukres}^{\Delta}_n}{G}$. Indeed, if $x, y \in S$, then there is $g\in G$ such that (1) $\Delta g\leq y$, by (\L12), $y\leq x\iml y$ \, and  so \,  $x\iml y \in S$. On the other hand, from (1), ($\Delta$\L11) and ($\Delta$\L16) \, we have \, $\Delta g\leq \Delta y$ \, and then \,  $\Delta y\in S$.
Finally, since $\Delta g\leq g$, for every $g\in G$, we have that $G\subseteq S$ and therefore $\free{\mbox{\lukres}^{\Delta}_n}{G}\subseteq S$. That is, $\free{\mbox{\lukres}^{\Delta}_n}{G} = S$ and then  $\mu \left(\free{\mbox{\lukres}^{\Delta}_n}{G}\right) \subseteq \Delta G$.  \\
Conversely, let $m\in\Delta G$, that is, $m=\Delta g$, for some $g\in G$ and $m\in\free{\mbox{\lukres}^{\Delta}_n}{G}$ . Suppose that there is  $t\in \free{\mbox{\lukres}^{\Delta}_n}{G}$ such that $t\leq m$. By the first part of this proof, $t\in S$, that is, $\Delta g' \leq t$ for some $g'\in G$ and then $\Delta g' \leq \Delta g$. By Lemma \ref{lemfree1}, we have $\Delta g' = \Delta g$, that is, $\Delta g =\Delta g' \leq t \leq m =\Delta g$, i.e., $t=m$. Then, $m\in \mu(\free{\mbox{\lukres}^{\Delta}_n}{G})$.
\end{dem}

\

\noindent From, Lemmas \ref{lemfree1}, \ref{free1} and Corollary \ref{corfree} we know that $\free{\mbox{\lukres}^{\Delta}_n}{G}$ is of the form sketched in Figure \ref{fig1}
  
 \begin{figure}[H]
  \begin{center}
  \begin{tikzpicture}[scale=.7]
\tikzstyle{every node}=[draw,circle,fill=white,inner sep=2pt]
  \node (one) at (2,3.5) [label=above:$1$] {};
  \node (g1) at (-2,1) [label=right:$g_1$] {};
 \node (dg1) at (-2,0) [label=right:$\Delta g_1$] {};
   \node (g2) at (0,1) [label=right:$g_2$] {};
 \node (dg2) at (0,0) [label=right:$\Delta g_2$] {};

\draw[black, dotted, ultra thick] (3,1) -- (3.5,1);
 
 \draw[black, dotted,ultra thick] (3,0) -- (3.5,0);


\draw[black, dotted, ultra thick] (2,2) -- (2,2.5);

  \draw (dg1) -- (g1) ;
\draw (dg2) -- (g2) ;
\end{tikzpicture}  
\end{center}
\caption{Lower part of the underlying lattice of $\free{\mbox{\lukres}^{\Delta}_n}{G}$}
\label{fig1}
\end{figure}

\noindent From Lemma \ref{free1}, we have:

\begin{corollary}\label{lemlibre1} \,  $\free{\mbox{\lukres}^{\Delta}_n}{G} = \bigcup \limits_{g\in G} \, \uparrow \left(\Delta g\right) =  \bigcup \limits_{g\in G} \, \left[\Delta g, 1\right]$.
\end{corollary}

\

For every class of algebras $\cal K$ and $\A, \B\in {\cal K}$, let us denote by $Hom_{\cal K}(\A,\B)$ and $Epi_{\cal K}(\A,\B)$ the set of all $\cal K$-homomorphisms and $\cal K$-epimorphisms, respectively,  from $\A$ into $\B$. Besides, if $X\subseteq A$ we denote by $[X]_{\cal K}$ the $\cal K$-subalgebra of $\A$ generated by $X$.   In \cite[pg. 135]{RS}, it was proved, for $2\leq k \leq n$,  that 
\begin{equation}\tag{RS}Epi_{{\bf W}_{n}}\left( \free{{\bf W}_{n}}{m}, {\bf \L}^{\bf W}_{k}\right) =v_m(k)= (k)^m  \, - \sum \limits_{\begin{array}{c } j-1\mid k-1 \\ j\not=k\end{array}} v_m(j)\end{equation}
Then,

\begin{lemma}\label{lemaux} \, $\left|Epi_{\mbox{\lukres}^{\Delta}_n}\left(\free{\mbox{\lukres}^{\Delta}_n}{m}, \mbox{{\bf \L}}^{\Delta}_{k}\right)\right| \leq \left|Epi_{{\bf W}_{n}}\left(\free{{\bf W}_{n}}{m}, \mbox{\bf \L}^{\bf W}_{k}\right)\right|$,  for $2\leq k \leq n$ .
\end{lemma}
\begin{dem} Let $G$ and $G'$ two sets of free generators such that $|G|=m=|G'|$ and let $\beta:G'\fun G$ a bijection. For every $h\in Epi_{\mbox{\lukres}^{\Delta}_n}\left(\free{\mbox{\lukres}^{\Delta}_n}{G}, \mbox{{\bf \L}}^{\Delta}_{k}\right)$ let $f:=\restr{h}{G}$ and let \break $f'=f\circ\beta:G'\fun \mbox{{\bf \L}}^{\Delta}_{k}$. Then, there exists a ${\bf W}_{n}$-homomorphism $h$, \break $h:\free{{\bf W}_{n}}{G'}\fun \mbox{{\bf \L}}^{\bf W}_{k}$ which extends $f'$. Then, $h'\in Epi_{{\bf W}_{n}}\left(\free{{\bf W}_{n}}{G'}, \mbox{\bf \L}^{\bf W}_{k}\right)$. Indeed, $h'\left(\free{{\bf W}_{n}}{G'}\right)=\left[f'(G')\right]_{{\bf W}_{n}}=\left[f(G)\right]_{{\bf W}_{n}}$, on the other hand,  $\mbox{\lukres}^{\Delta}_{k} = h(\free{\mbox{\lukres}^{\Delta}_n}{m})=\left[f(G)\right]_{\mbox{\lukres}^{\Delta}_n} \subseteq \left[f(G)\right]_{{\bf W}_{n}}$, then $h'(\free{{\bf W}_{n}}{m})=\mbox{\bf \L}^{\bf W}_{k}$.
Let 
$$\psi: Epi_{\mbox{\lukres}^{\Delta}_n}\left(\free{\mbox{\lukres}^{\Delta}_n}{G}, \mbox{{\bf \L}}^{\Delta}_{k}\right) \fun Epi_{\bf W}\left(\free{{\bf W}_{n}}{G'}, \mbox{\bf \L}^{\bf W}_{k}\right)$$

defined by \, $\psi(h)=h'$. Then, it is not difficult to check that $\psi$ is one-to-one.
\end{dem}

\

\begin{lemma} \, $\free{\mbox{\lukres}^{\Delta}_n}{m}$ is finite,  for  $2\leq k \leq n$.
\end{lemma}
\begin{dem} By Lemma \ref{lemprodsub}, there exists a $\mbox{\lukres}^{\Delta}_n$-monomorphism 
$$i:\free{\mbox{\lukres}^{\Delta}_n}{m}\fun \prod \limits_{D\in \E_{\lukres^{\Delta}_n}\left(\free{\mbox{\lukres}^{\Delta}_n}{m}\right)} \free{\mbox{\lukres}^{\Delta}_n}{m}/D.$$
and by Corollary \ref{lemsim}, $\free{\mbox{\lukres}^{\Delta}_n}{m}/D \cong \mbox{{\bf \L}}^{\Delta}_{k}$  and $D\in \E_{\lukres^{\Delta}_n}(\free{\mbox{\lukres}^{\Delta}_n}{m})$. Then, it is enough to prove that $\E_{\lukres^{\Delta}_n}(\free{\mbox{\lukres}^{\Delta}_n}{m})$ is a finite set. Indeed, for every \break $h\in Hom_{\mbox{\lukres}^{\Delta}_n}\left(\free{\mbox{\lukres}^{\Delta}_n}{m}, \mbox{{\bf \L}}^{\Delta}_{k}\right)$ we have that $Ker(h)=\{x:h(x)=1\}\in  \E_{\lukres^{\Delta}_n}(\free{\mbox{\lukres}^{\Delta}_n}{m})$ and the map $h \xmapsto[]{} Ker(h)$ is onto. The proof ends noting that 
 $$Hom_{\mbox{\lukres}^{\Delta}_n}\left(\free{\mbox{\lukres}^{\Delta}_n}{m}, \mbox{{\bf \L}}^{\Delta}_{k}\right) = \bigcup \left\{Epi_{\mbox{\lukres}^{\Delta}_n}\left(\free{\mbox{\lukres}^{\Delta}_n}{m}, S\right) : S \mbox{ is a subalgebra of } \mbox{{\bf \L}}^{\Delta}_{k}\right\},$$
 the sets $Epi_{\mbox{\lukres}^{\Delta}_n}\left(\free{\mbox{\lukres}^{\Delta}_n}{m}, S\right)$, for $S$ subalgebra of $\mbox{{\bf \L}}^{\Delta}_{k}$, are finite (Lemma \ref{lemaux}) and that  $\mbox{{\bf \L}}^{\Delta}_{k}$ has a finite number of subalgebras.
\end{dem}

\

\begin{corollary} The variety $\mbox{\lukres}^{\Delta}_n$ is locally finite.
\end{corollary}

\

\noindent Let $G=\{g_1,\dots,g_m\}$ be a finite set of free generators. By Theorem \ref{lemlibre1} and the {\em principle of inclusion--exclusion}, we have 
\begin{equation} \left|\free{\mbox{\lukres}^{\Delta}_n}{m}\right| = \left| \bigcup \limits_{i=1}^{m} \, \uparrow \left(\Delta g_i\right) \right| = \sum \limits_{k=1}^{m} (-1)^{k+1} \,  a_k
\end{equation} 
where
\begin{equation} a_k = \sum \limits_{1\leq i_1\leq\dots\leq i_k\leq m} \left| \bigcap \limits_{t=1}^{k} \uparrow \left(\Delta g_{i_t}\right) \right|
\end{equation} 

\noindent By the symmetry of the problem, it is enough to determine  $\left| \bigcap \limits_{j=1}^{k} \uparrow \left(\Delta g_{j}\right) \right|$. \\
Let $G_k=\{g_1, \dots, g_k\}$ and $G_{m-k}=\{g_{k+1}, \dots, g_m\}$ and $g_k^*=\bigvee \limits_{j=1}^{k} g_j$. Then, 
$$\bigcap \limits_{j=1}^{k} \uparrow \left(\Delta g_{j}\right) = N_k = \left\{ x\in \free{\mbox{\lukres}^{\Delta}_n}{m} : \Delta g_k^*\leq x \right\}$$
and so far, we have
\begin{equation}\left|\free{\mbox{\lukres}^{\Delta}_n}{m}\right| = \sum \limits_{k=1}^{m} (-1)^{k+1} \binom{m}{k} \, |N_k|
\end{equation} 

\

\noindent It is not difficult to check that $N_k$ is a $\mbox{\lukres}^{\Delta}_n$-subalgebra of $\free{\mbox{\lukres}^{\Delta}_n}{m}$ with first element $\Delta g_k^*$. Then, the algebra $\bar{N}_k=\langle N_k, \iml, \Delta g_k^*, 1\rangle \in \mbox{\lukres}^{0}_n$  and ${\bf W}(\bar{N}_k)$ is a finite ${\bf W}_{n}$-algebra. By \cite[pg. 97 and pg. 132]{RS}, we have

\begin{equation}\label{eq1} {\bf W}(\bar{N}_k) \cong \prod \limits_{D\in \E(N_k)} {\bf W}(\bar{N}_k)/D
\end{equation} 
Let $i$, $2\leq i\leq n$, and let $\E_{\lukres^{\Delta}_{i}}(N_k)=\left\{D\in \E_{\lukres^{\Delta}_{n}}(N_k) : \bar{N}_k/D \cong \mbox{{\bf \L}}^{\bf W}_{i} \right\}$. It is easy to verify that $\E_{\lukres^{\Delta}_{i}}(N_k)\not=\emptyset$. Then, let
\begin{equation}\label{eq2} \beta_{i+1}(k) = |\E_{\lukres^{\Delta}_{i}}(N_k)|
\end{equation} 

\noindent From (\ref{eq1}) and  (\ref{eq2}) 
\begin{equation}\label{eq3} |N_k|=\prod \limits_{i=2}^{n} i^{\beta_{i}}
\end{equation} 
From \cite{GPS2020}, we know that 
\begin{equation}\label{eq3} \beta_{i}(k)= (i-1)^k  \, i^{m-k} - \sum \limits_{\begin{array}{c }j-1 \mid k-1 \\ j\not=k\end{array}} \beta_{j}(k) 
\end{equation} 
Finally 
\begin{equation} \left|\free{\mbox{\lukres}^{\Delta}_n}{m}\right| = \sum \limits_{k=1}^{m} (-1)^{k+1} \binom{m}{k} \prod \limits_{i=2}^{n} i^{\beta_{i}(k)}
\end{equation} 

\

\section{$n$-valued super-\L ukasiewicz logic expanded by $\Delta$}

In this section, we present a family of $n$-valued logics ($n\geq 2$), by means of Hilbert-style systems which constitute a generalization of the  system studied in \cite{AVF}. Let us consider a denumerable set $Var$ of propositional variables, and let $Fm$ be the propositional language generated by $Var$ over the signature $\{\iml, \Delta\}$, and $\mathfrak{Fm}$ the free algebra with the signature $\{\iml, \Delta, \top\}$. As usual, lowercase Greek letters stand for propositional formulas and uppercase Greek letters stand for sets of formulas. Now, consider the non primitive connective $\imn$ defined by 
$$ \alpha \rightarrow_0 \beta := \beta, \hspace{1cm} \alpha \rightarrow_{n} \beta := \alpha \iml (\alpha \imn \beta), \mbox{ if } n>0 $$


\noindent For $n\geq 2$, the Hilbert-style calculus $\luklog_{n}$ is defined as follows: \\[2mm]
{\bf Axiom schemata:} 

\vspace{-0.3cm}
\begin{myequation}
  \alpha\iml(\beta\iml\alpha)
   \end{myequation}
\vspace{-0.5cm}
\begin{myequation}
 (\alpha\iml \beta)\iml((\beta\iml \gamma)\iml(\alpha\iml\gamma))
   \end{myequation}
\vspace{-0.5cm}
\begin{myequation}
((\alpha\iml \beta)\iml \beta)\iml((\beta\iml\alpha)\iml \alpha)
   \end{myequation}
\vspace{-0.5cm}
\begin{myequation}
((\alpha\iml \beta)\iml(\beta\iml \alpha))\rightarrow(\beta\iml \alpha)
   \end{myequation}
\vspace{-0.5cm}
\begin{equation}\tag{{\bf Ax}5$n$}
((\alpha\imn \beta)\iml \alpha)\iml \alpha
   \end{equation}
\vspace{-0.5cm}
\begin{equation}\tag{{\bf Ax}6}
(\Delta \alpha\iml\Delta \beta)\iml\Delta(\Delta \alpha\iml \beta)
   \end{equation}
   \vspace{-0.5cm}
\begin{equation}\tag{{\bf Ax}7$n$}
\Delta(\Delta \alpha\iml \beta)\iml(\alpha\imn \Delta \beta)
    \end{equation}
   \vspace{-0.5cm}
\begin{equation}\tag{{\bf Ax}8$n$}
(\alpha\imn \beta)\iml(\Delta \alpha\iml \beta)
    \end{equation}
and the only inference rule is {\em modus ponens}: 
\begin{equation}\tag{MP}
       \displaystyle\frac{\alpha ,\quad \alpha\rightarrow \beta}{\beta}
\end{equation}

\noindent We also  consider the non-primitive connective $\vee$ defined as follows: $$\alpha \vee \beta:= (\alpha \rightarrow \beta)\rightarrow \beta$$
Then, ({\bf Ax}3)--({\bf Ax}5) can be expressed, respectively,  as \, 
$$(\alpha\vee\beta)\iml(\beta\vee\alpha), \hspace{0.5cm} (\alpha\iml\beta)\vee(\beta\iml\alpha), \, \mbox{ and \,  } \,  \, (\alpha\imn\beta)\vee\alpha.$$
The notions of (syntactic) theorem, derivation of a formula from a set of hypothesis and derived rule of inference  are the usual. Besides, we write  $\Gamma \vdash_{\luklog_n} \alpha$ to indicate that there is a derivation of $\alpha$ from $\Gamma$ in $\luklog_n$. It $\alpha$ is a theorem of $\luklog_n$ we write $\vdash_{\luklog_n} \alpha$. The calculus $\luklog_{n}$ gives rise to a propositional logic which we name as the calculus. The following lemma shows some theorems and derived rules of the $\{\iml\}$-fragment of $\luklog_n$.

\

\begin{proposition} The following formulas (rules) are theorems (derived) in $\luklog_n$. 
\hspace{0.5cm}
 \begin{multicols}{2}
\begin{itemize}
      \item[($\luklog_n$1)]  $((\alpha\iml \beta)\iml \gamma)\iml(\beta\iml \gamma)$;
    \item[($\luklog_n$2)] $\displaystyle\frac{\alpha\iml \beta, \beta\iml\gamma}{\alpha\iml\gamma}$;
    \item[($\luklog_n$3)] $\alpha\iml \alpha\vee \beta$;
    \item[($\luklog_n$4)] $((\alpha\vee \gamma)\iml \beta)\iml(\alpha\iml \beta)$;
    \item[($\luklog_n$5)]  $\alpha\iml \alpha$;
    \item[($\luklog_n$6)]  $\big(((\beta\iml\beta)\iml\alpha)\iml\alpha\big)$;
    \item[($\luklog_n$7)]  $(\alpha\iml (\beta\iml \gamma))\iml(\beta\iml(\alpha\iml \gamma))$;
    \item[($\luklog_n$7')]  $\displaystyle\frac{\alpha\iml (\beta\iml \gamma)}{\beta\iml(\alpha\iml \gamma)}$;
    \item[($\luklog_n$8)] $\beta\iml(\alpha\iml \alpha)$;
    \item[($\luklog_n$9)] $((\alpha\vee \gamma)\iml(\beta\iml \gamma))\iml(\alpha\iml(\beta\iml\gamma))$;
      \item[($\luklog_n$10)] $\displaystyle\frac{\alpha\iml\beta}{(\gamma\iml\alpha)\iml(\gamma\iml\beta)}$;
  \item[($\luklog_n$10')] $\displaystyle\frac{\alpha\iml\beta}{(\beta\iml\gamma)\iml(\alpha\iml\gamma)}$;
    \item[($\luklog_n$11)] $(\alpha\iml(\beta\iml\gamma))\iml((\beta\vee \gamma)\iml(\alpha\iml\gamma))$.
\end{itemize}
 \end{multicols}
\end{proposition}
\begin{proof} It is routine.
\end{proof}

\ 

\noindent Note that the rule {\em $k$-weak modus ponens} 
\begin{equation}\tag{$k$WMP}
\displaystyle\frac{\alpha, \, \alpha\imk\beta}{\beta}
\end{equation}
is derivable in $\luklog_n$ every $k\geq 1$.\\[2mm]
The following proposition shows some theorems and derived rules of $\luklog_n$ involving $\imn$. We shall note \, $\alpha\ssi\beta$ \, to indicate that both $\alpha\iml\beta$ and $\beta\iml\alpha$ are theorems.

\

\begin{proposition} The following formulas (rules) are theorems (derived) in $\luklog_n$.
\hspace{0.5cm}
\begin{itemize}
    \item[($\luklog_n$12)]  $\displaystyle\frac{\alpha\iml \beta}{(\gamma\imk\alpha)\iml(\gamma\imk\beta)}$; 
    \item[($\luklog_n$13)] $(\alpha \imk(\beta\iml\gamma))\ssi(\beta\iml(\alpha\imk\gamma))$;
    \item[($\luklog_n$14)] $\alpha\imk\alpha$;
    \item[($\luklog_n$15)] $(\alpha\imn(\alpha\iml \beta))\ssi(\alpha\imn\beta)$;
  \item[($\luklog_n$15')] $(\alpha\imn(\alpha\imn \beta))\ssi(\alpha\imn\beta)$;
   \item[($\luklog_n$16)]  $(\alpha\imn(\beta\iml\gamma))\ssi((\alpha\imn\beta)\iml(\alpha\imn\gamma))$;
   \item[($\luklog_n$17)] $(\alpha\imn(\beta\imk\gamma))\ssi((\alpha\imn\beta)\imk(\alpha\imn\gamma))$;
 \item[($\luklog_n$18)]  $\displaystyle\frac{\alpha\imn\beta, \, \beta\imn\gamma}{\alpha\imn\gamma}$,
 \item[($\luklog_n$19)] $\displaystyle\frac{\alpha\imk \beta}{(\gamma\iml\alpha)\imk(\gamma\iml\beta)}$ \, and \, $\displaystyle\frac{\alpha\imk\beta}{(\beta\iml\gamma)\imk(\alpha\iml\gamma)}$.
\end{itemize}
\end{proposition}
\begin{proof} \\
($\luklog_n$12): \,  For $k=1$, it holds by ($\luklog_n$10).

\

\begin{tabular}{ll|r}
1. & $\alpha\iml\beta$ & Hyp. \\ 
2. & $(\gamma\imk\alpha)\iml(\gamma\imk\beta)$ & 1, (I.H.) \\ 
3. & $(\gamma\iml(\gamma\imk\alpha))\iml(\gamma\iml(\gamma\imk\beta)$ & 2,($\luklog_n$10)\\ 
4.&$(\gamma\rightarrow_{k+1}\alpha))\iml(\gamma\rightarrow_{k+1}\beta)$ & 3\\ 
\end{tabular} \\[3mm]
\noindent ($\luklog_n$13): \,  For $k=1$, it holds by ($\luklog_n$7).\\[2mm]

\

\begin{tabular}{ll|r}
1. & $(\alpha\imk(\beta\iml\gamma))\iml(\beta\iml(\alpha\imk\gamma))$ &  (I.H.) \\ 
2. & $\big(\alpha\iml(\alpha\imk(\beta\iml\gamma))\big)\iml\big(\alpha\iml(\beta\iml(\alpha\imk\gamma))\big)$ & 1, ($\luklog_n$10) \\ 
3. & $\big(\alpha\iml(\beta\iml(\alpha\imk\gamma))\big)\iml\big(\beta\iml(\alpha\iml(\alpha\imk\gamma))\big)$ & 2,($\luklog_n$7)\\ 
4.&$(\alpha\rightarrow_{k+1}(\beta\iml\gamma))\iml(\beta\iml(\alpha\rightarrow_{k+1}\gamma))$ & 2, 3 and ($\luklog_n$2) \\ 
\end{tabular} \\[6mm]
$(\alpha \iml(\beta\imk\gamma))\iml(\beta\imk(\alpha\iml\gamma))$ \, is proved analogously. \\[2mm]
 ($\luklog_n$14): From ($\luklog_n$5) and ({\bf Ax}1).\\[2mm]
($\luklog_n$15): From ({\bf Ax}5$n$), ({\bf Ax}3), ($\luklog_n$13) and (MP). \\[2mm]
($\luklog_n$16): 

\
\begin{center}
\begin{tabular}{ll|r}
1. & $(\alpha\iml\beta)\iml((\beta\iml\gamma)\iml(\alpha\iml\gamma))$ &  ({\bf Ax}2) \\ 
2. & $(\beta\iml\gamma)\iml((\alpha\iml\beta)\iml(\alpha\iml\gamma))$ & 1, ($\luklog_n$7') \\ 
3. & $\big(\alpha\imn(\beta\iml\gamma)\big)\iml\big(\alpha\imn((\alpha\iml\beta)\iml(\alpha\iml\gamma))\big)$ & 2, ($\luklog_n$12) \\ 
4. & $\big(\alpha\imn((\alpha\iml\beta)\iml(\alpha\iml\gamma))\big)\iml\big((\alpha\iml\beta)\iml(\alpha\imn(\alpha\iml\gamma))\big)$ & ($\luklog_n$13) \\ 
5. & $\big(\alpha\imn(\beta\iml\gamma)\big)\iml\big((\alpha\iml\beta)\iml(\alpha\imn(\alpha\iml\gamma))\big)$ & 3, 4,   ($\luklog_n$2) \\ 
6. & $(\alpha\imn(\alpha\iml \gamma))\iml(\alpha\imn\gamma)$ &  ($\luklog_n$15) \\ 
7. & $\big((\alpha\iml\beta)\iml(\alpha\imn(\alpha\iml \gamma))\big)\iml\big((\alpha\iml\beta)\iml(\alpha\imn\gamma)\big)$ &  6, ($\luklog_n$12) \\ 
8. & $\big(\alpha\imn(\beta\iml\gamma)\big)\iml \big((\alpha\iml \beta)\iml(\alpha\imn\gamma)\big)$ &  5, 7 and ($\luklog_n$2) \\
9. & $(\alpha\iml \beta)\iml \big(\big(\alpha\imn(\beta\iml\gamma)\big)\iml(\alpha\imn\gamma)\big)$ &  8, ($\luklog_n$7') \\
10. & $\big(\alpha\imn(\alpha\iml \beta)\big)\iml \big(\alpha\imn\big((\alpha\imn(\beta\iml\gamma))\iml(\alpha\imn\gamma)\big)\big)$ &  9,  ($\luklog_n$13) \\
11. & $(\alpha\imn\beta)\iml(\alpha\imn(\alpha\iml \beta))$ &  ($\luklog_n$15) \\
12. & $(\alpha\imn\beta)\iml\big(\alpha\imn\big((\alpha\imn(\beta\iml\gamma))\iml(\alpha\imn\gamma)\big)\big)$ &  10, 11,  ($\luklog_n$2) \\
13. & $\big(\alpha\imn\big((\alpha\imn(\beta\iml\gamma))\iml(\alpha\imn\gamma)\big)\big)\iml$ & \\
&  $\big((\alpha\imn(\beta\iml\gamma))\iml\big(\alpha\imn(\alpha\imn\gamma)\big)\big)$ &   ($\luklog_n$13) \\
14. & $(\alpha\imn\beta)\iml \big((\alpha\imn(\beta\iml\gamma))\iml\big(\alpha\imn(\alpha\imn\gamma)\big)\big)$ &  12, 13, ($\luklog_n$2) \\
15. & $(\alpha\imn(\beta\iml\gamma))\iml \big((\alpha\imn\beta)\iml(\alpha\imn(\alpha\imn\gamma))\big)$ &  14, ($\luklog_n$7') \\
16. & $((\alpha\imn\beta)\iml(\alpha\imn(\alpha\iml \gamma)))\iml((\alpha\imn\beta)\iml(\alpha\imn\gamma))$ &  15, ($\luklog_n$10) \\
17. & $(\alpha\imn(\beta\iml\gamma))\iml((\alpha\imn\beta)\iml(\alpha\imn\gamma))$ & 15, 16, ($\luklog_n$2)
\end{tabular}
\end{center}
The other implication is proved analogously.\\[2mm]
($\luklog_n$17): \,  From ($\luklog_n$16) and using induction on $k$.\\[3mm]
($\luklog_n$18): 

\

\begin{tabular}{ll|r}
1. & $\alpha\imn\beta$ &  Hyp. \\ 
2. & $\beta\imn\gamma$ & Hyp. \\ 
3. & $\alpha\imn((\beta\imn\gamma)\iml(\beta\imn\gamma))$ &($\luklog_n$5), ({\bf Ax}1), (MP) \\ 
4. & $(\beta\imn\gamma)\iml(\alpha\imn(\beta\imn\gamma))$ & ($\luklog_n$13), (MP) \\ 
5. & $\alpha\imn(\beta\imn\gamma)$ & 2, 4, (MP) \\ 
6. & $\alpha\imn(\beta\imn\gamma)\iml((\alpha\imn\beta)\imn(\alpha\imn\gamma))$ & ($\luklog_n$17), $k=n$ \\ 
7. & $(\alpha\imn\beta)\imn(\alpha\imn\gamma)$ & 5, 6, (MP)\\ 
8. & $(\alpha\imn\gamma)$ & 1, 7, ($n$WMP)\\ 
\end{tabular} \\[2mm]
($\luklog_n$19): By induction on $k$.
\end{proof}

\

\begin{lemma}
The following formulas (rules) are theorems (derived) in $\luklog_n$.
 \begin{multicols}{2}
\begin{enumerate}
      \item[($\luklog_n$20)] \, $\Delta \alpha\iml \alpha$;
      \item[($\luklog_n$21)] \, $\alpha\rightarrow_{n}\Delta \alpha$;
      \item[($\luklog_n$22)] \, $\displaystyle\dfrac{\alpha}{\Delta \alpha}$;    
      \item[($\luklog_n$23)] \, $\Delta(\Delta \alpha\iml \alpha)$;
      \item[($\luklog_n$24)] \, $\Delta(\Delta \alpha\iml \beta)\iml(\Delta \alpha\iml \Delta \beta)$;
      \item[($\luklog_n$25)] \, $\displaystyle\dfrac{\alpha\iml \beta}{\Delta \alpha\iml\Delta \beta}$;
     \item[($\luklog_n$26)] \, $\displaystyle\dfrac{\alpha\rightarrow_{n} \beta}{\Delta \alpha\rightarrow_{n}\Delta \beta}$;
  
  \item[($\luklog_n$27)] \, $(\Delta \alpha\iml \beta)\iml(\alpha\imn \beta)$.
  \end{enumerate}
\end{multicols}
\end{lemma}
\begin{proof} \, ($\luklog_n$20):

\

\begin{tabular}{ll|r}
1. & $(\alpha\rightarrow_{n} \alpha)\iml(\Delta \alpha\iml \alpha)$ & ({\bf Ax}8$n$) \\
2. & $\alpha\rightarrow_{n}\alpha$ & ($\luklog_n$14)  \\
4. & $\Delta \alpha\iml \alpha$ & 2, 1 and (MP)  \\
\end{tabular}\\[3mm]
($\luklog_n$21):
\

\begin{tabular}{ll|r}
1. & $\alpha$ & hyp. \\
2. & $\Delta \alpha\iml\Delta \alpha $ & ($\luklog_n$5) \\
3. & $\Delta (\Delta \alpha\iml\alpha) \iml ( \alpha \rightarrow_{n}  \Delta\alpha)$ & ({\bf Ax}7$n$) \\
 4. & $(\Delta \alpha \iml \Delta \alpha) \iml( \Delta(\Delta \alpha \iml \alpha))$ &  ({\bf Ax}6)\\
5. & $\Delta(\Delta \alpha \iml \alpha)$ &  2, 4 and (MP)\\
6.  & $\alpha \rightarrow_{n}  \Delta\alpha$ & 5, 3 and (MP)\\
\end{tabular} \\[3mm]
($\luklog_n$22): It is consequence of ($\luklog_n$21) and (MP).\\[3mm]
($\luklog_n$23): It is consequence of ($\luklog_n$20) and ($\luklog_n$22).\\[3mm]
($\luklog_n$24):

\begin{tabular}{ll|r}
1. & $(\alpha \rightarrow_{n} \Delta \beta)\iml(\Delta \alpha \iml\Delta \beta)$ & ({\bf Ax}8$n$) \\
2. & $\big(\Delta(\Delta \alpha\iml \beta)\iml(\alpha \rightarrow_{n} \Delta \beta)\big)\iml\big(\Delta(\Delta \alpha\iml \beta)\iml(\Delta \alpha \iml\Delta \beta)\big)$ &   1 and  ($\luklog_n$10) \\
3. & $\Delta(\Delta \alpha \iml \beta)\iml(\alpha \rightarrow_{n} \Delta \beta))$ & ({\bf Ax}7$n$)  \\
4. & $\Delta(\Delta \alpha\iml \beta)\iml(\Delta \alpha\iml\Delta \beta)$ & 3, 2 and  (MP) \\
\end{tabular}\\[3mm]
($\luklog_n$25):
\

\begin{tabular}{ll|r}

1. & $\alpha \iml \beta $ & Hyp. \\
2. & $(\Delta \alpha\iml \beta)\iml(\Delta \alpha\iml \beta)$ & ($\luklog_n$5$n$) \\
3. & $\Delta\alpha \iml\alpha$ & ($\luklog_n$25)  \\
4. & $(\alpha\iml\beta)\iml(\Delta\alpha \iml\beta)$ & 3 and ($\luklog_n$10')  \\
5. & $\Delta \alpha\iml \beta$ & 1, 4  and (MP)  \\
6. & $\Delta(\Delta \alpha\iml \beta)$ & 5 and ($\luklog_n$22) \\
7. & $\Delta(\Delta \alpha\iml \beta)\iml(\Delta \alpha\iml\Delta \beta)$ &($\luklog_n$24) \\
8. & $\Delta \alpha\iml\Delta \beta$ & 5, 6  and  (MP) \\
\end{tabular}\\[3mm]
($\luklog_n$26):
\

\begin{tabular}{ll|r}

1. & $\alpha \imn \beta $ & Hyp. \\
2. & $(\alpha\imn \beta)\iml(\Delta \alpha\iml \beta)$ & ({\bf Ax}8$n$) \\
3. & $\Delta\alpha \iml\beta$ & 1, 2 and (MP)   \\
4. & $\Delta\Delta\alpha \iml\Delta\beta$ & 3 and ($\luklog_n$25)  \\
5. & $\Delta\big(\Delta\Delta\alpha \iml\Delta\beta\big)$ & 4  and ($\luklog_n$22)  \\
6. & $\Delta\big(\Delta\Delta\alpha \iml\Delta\beta\big)\iml\big(\Delta\alpha\imn\Delta\beta\big)$ & ({\bf Ax}7$n$) \\
7. & $\Delta\alpha\imn\Delta\beta$ & 5, 6 and (MP) \\
\end{tabular}\\[3mm]
($\luklog_n$27):
\

\begin{tabular}{ll|r}
 1. & $\alpha \imn \Delta \alpha$ & ($\luklog_n$21) \\
2. & $(\beta\iml \Delta \alpha)\iml(\alpha \imn \Delta \alpha)$ & 1, ({\bf Ax}1), (MP)\\
3.  & $\alpha\iml( (\beta\iml\Delta \alpha)\iml(\alpha \imn \Delta\alpha))$  & 2.  and (T6)\\
4. & $( (\beta\iml\Delta \alpha)\iml(\alpha \imn\alpha))\iml ( \alpha\imn((\beta\iml\Delta \alpha)\iml \Delta\alpha))$  &  ($\luklog_n$13)\\
5. & $\alpha\iml( \alpha\imn((\beta\iml\Delta \alpha)\iml \Delta\alpha)) $  & 3, 4,  ($\luklog_n$2)\\

6. & $((\beta\iml\Delta\alpha)\iml\Delta \alpha)\iml((\Delta\alpha\iml \beta)\iml \beta)$ &  ({\bf Ax}3)\\

7. & $(\alpha \imn ((\beta\iml\Delta \alpha)\iml\Delta \alpha)) \iml( \alpha \imn((\Delta \alpha\iml\beta)\iml\beta))$ &   6,  ($\luklog_n$12) \\
8. & $\alpha \imn((\Delta \alpha\iml\beta)\iml\beta)$ & 5, 7, (MP) \\
9. & $(\alpha \imn((\Delta \alpha\iml\beta)\iml\beta))\iml((\Delta \alpha\iml\beta)\iml(\alpha \imn\beta))$ &   ($\luklog_n$1)\\
10.  & $(\Delta \alpha\iml\beta)\iml(\alpha \imn\beta) $ & 8, 9, (MP)\\
\end{tabular}

\end{proof}

\

\

Next, we shall construct the {\em Lindenbaum-Tarski algebra} for the calculus  for $\luklog_n$, $n\geq 2$. So, let $\alpha, \beta\in Fm$ then let us consider the relation $\equiv$ defined by
$$\alpha\equiv\beta \, \mbox{ \, iff \, } \vdash_{\luklog_n} \alpha \imn \beta \mbox{ and }\vdash_{\luklog_n} \beta \imn \alpha$$ 
Then, 

\

\begin{lemma} \,  $\equiv$ is a congruence relation on $\mathfrak{Fm}$.
\end{lemma}
\begin{proof} From the definition of $\equiv$, ($\luklog_n$14) and ($\luklog_n$18) we have that $\equiv$ is a equivalence relation on $Fn$. Besides, if $\alpha,\beta\in Fm$ are such that $\alpha\equiv\beta$ then, by ($\luklog_n$26), we have $\vdash_{\luklog_n} \Delta\alpha \imn \Delta\beta$ and therefore, we have \, $ \Delta\alpha \equiv \Delta\beta$. Besides, if $\gamma,\delta\in Fm$ are such that $\gamma\equiv\delta$ then, from $\vdash_{\luklog_n} \gamma \imn \delta$ and  ($\luklog_n$19) we have (1) $\vdash_{\luklog_n} (\beta \iml \gamma)\imn(\beta\iml \delta)$. On the other hand, from $\vdash_{\luklog_n} \alpha \imn \beta$  and  ($\luklog_n$19) we have (2) $\vdash_{\luklog_n} (\alpha\iml\gamma) \imn (\beta\iml\gamma)$. From (1), (2), ($\luklog_n$18), we get   $\vdash_{\luklog_n} (\alpha\iml\gamma) \imn (\beta\iml\delta)$. In a similar way, we prove that  $\vdash_{\luklog_n} (\beta\iml\delta) \imn(\alpha\iml\gamma)$. Therefore, \, $ (\alpha\iml\gamma) \equiv (\beta\iml\delta)$.
\end{proof}\\[2mm]
If $\alpha\in Fm$, we denote the class of $\alpha$ determine by $\equiv$ by $\overline{\alpha}$.

\begin{theorem} \label{linden}
The Lindenbaum-Tarski algebra $\mathfrak{Fm}/\equiv=\langle Fm/\equiv, \iml, \Delta, 1\rangle $ is a \lukres$^{\Delta}_n$-algebra  where $\overline{\alpha \iml \beta}  = \overline{\alpha} \iml \overline{\beta}$, $\overline{\Delta \alpha}=\Delta \overline{\alpha}$  and $1=\overline{\alpha \iml \alpha}=\{\phi\in Fm:\quad\vdash_{\luklog_n}\phi\}$. Moreover, the relation $\overline{\alpha} \leq \overline{\beta}$, defined by $\vdash_{\luklog_n} \alpha \rightarrow_{n} \beta$, is a partial order on $Fm/\equiv$. 
\end{theorem}
\begin{proof} It is easy to check that \, $\leq$ \, is a partial order on $ Fm/\equiv$. Besides, it is clear that  $\overline{\beta}\leq \overline{\alpha\rightarrow\alpha}=1$, for every $\alpha$ and $\beta$.\\
Let us verify that $\langle Fm/\equiv, \rightarrow, \Delta, 1\rangle $ is an  \lukres$^{\Delta}_n$-algebra. From ({\bf Ax}1)--({\bf Ax}5) we have that $\langle Fm/\equiv, \rightarrow, 1\rangle$ is a  \lukres$_n$-algebra. Besides from ({\bf Ax}8) and ($\luklog_n$27), on the one hand and, ({\bf Ax}6) and ($\luklog_n$24), on the other, we have that $\langle For/\equiv, \rightarrow, \Delta, 1\rangle $ verifies ($\Delta$\L 1) and ($\Delta$\L 2).
\end{proof}

\ 

\noindent Let $\A$ be an $\lukres^{\Delta}_n$-algebra. As usual, we call {\em valuation} (on $\A$) to any function $v:Fm\fun A$ which preserve the operations $\iml$ and $\Delta$. With a slight abuse of notation, we say that $v$ is a valuation if
$$v\in Hom_{\lukres^{\Delta}_n}\big( \mathfrak{Fm}, \A \big).$$
Let $\alpha \in Fm$, we say that $\alpha$ {\em is valid on $\A$} if $v(\alpha)=1$ for any valuation $v$ on $\A$; and we say that $\alpha$ is  {\em semantically valid}, denoted by \, $\vDash_{\lukres^{\Delta}_n}\alpha$, if $\alpha$ is valid on every $\lukres^{\Delta}_n$-algebra. Then,

\begin{theorem}\label{wat}{\rm (Weak Adequacy Theorem).}\label{WAT} Let $\alpha\in Fm$. Then, 
$$\vDash_{\luklog_n}\alpha \, \mbox{ if and only if } \, \vdash_{\lukres^{\Delta}_n}\alpha$$
\end{theorem}
\begin{proof}
({\bf Soundness part}): It is a routine task to check that every axiom of $\luklog_n$ is semantically valid; and that the rule (MP) preserves validity, that is, if $\vDash_{\luklog_n}\alpha$ and $\vDash_{\luklog_n}\alpha\iml\beta$, then $\vDash_{\luklog_n}\beta$. Then, the proof goes by induction on the length (number of steps) of the deduction of $\alpha$.\\[1mm]
({\bf Completeness part}): Let $\alpha\in Fm$ such that $\alpha$ is semantically valid. Then, for every  $\lukres^{\Delta}_n$-algebra $\A$ and every valuation $v:\mathfrak{Fm}\fun \A$ we have $v(\alpha)=1$.
In particular, if we take \, $\mathfrak{Fm}/\equiv$ \, and consider the natural projection $\pi:\mathfrak{Fm} \fun \mathfrak{Fm}/\equiv$ defined by $\pi(\gamma)=\overline{\gamma}$ (the class of $\gamma$ by $\equiv$), it is clear that $\pi$ is a valuation.  Therefore, $\pi(\alpha)=\overline{\alpha}=1=\{\beta \in Fm :\, \vdash_{\luklog_n} \beta\}$ and so $\vdash_{\luklog_n}\alpha$ as desired.
\end{proof}

\

\noindent We end this section showing the relation between the systems $\luklog_{n}$ for $n\geq2$.

\begin{lemma} For every $n\geq 1$, the logic $\luklog_{n+1}$ is a proper sub-logic of $\luklog_{n}$. 
\end{lemma}
\begin{proof} All the theorems of $\luklog_{n+1}$ are also theorems of $\luklog_{n}$. Indeed, axioms ({\bf Ax}1)--({\bf Ax}4) hold in both systems. Besides, from ({\bf Ax}1) we have that $(\alpha \rightarrow_{n-1} \beta)\iml(\alpha \rightarrow_{n} \beta)$; and applying  ($\luklog_n$10') twice we get 
$$\big(((\alpha \rightarrow_{n-1} \beta)\iml\alpha)\iml\alpha\big)\iml\big(((\alpha \rightarrow_{n} \beta)\iml\alpha)\iml\alpha\big)$$
That is, from ({\bf Ax}5$n$) and (MP) we have ({\bf Ax}5$(n+1)$). Analogously, from  ({\bf Ax}7$n$) we obtain ({\bf Ax}7$(n+1)$).
On the other hand,\\[2mm]
\begin{tabular}{ll|r}
1. & $((\alpha \rightarrow_{n-1} \beta)\iml\alpha)\iml\alpha$ & ({\bf Ax}5$n$) \\
2. & $\big(((\alpha \rightarrow_{n-1} \beta)\iml\alpha) \iml\alpha\big)\iml\big((\alpha\iml(\alpha \rightarrow_{n-1} \beta))\iml(\alpha \rightarrow_{n-1} \beta)\big)$ & ({\bf Ax}3) \\
3. & $(\alpha\iml(\alpha \rightarrow_{n-1} \beta))\iml(\alpha \rightarrow_{n-1} \beta)$ & 1, 2 and (MP)   \\
3. & $(\alpha \rightarrow_{n} \beta)\iml(\alpha \rightarrow_{n-1} \beta)$ &    \\
\end{tabular}\\[3mm]
that is, in $\luklog_{n}$ it holds $(\alpha \rightarrow_{n} \beta)\iml(\alpha \rightarrow_{n-1} \beta)$, and using ({\bf Ax}8$n$) and (MP), we obtain ({\bf Ax}8$(n+1)$).\\[2mm]
Finally, consider the $\lukres^{\Delta}_{n+1}$-algebra {\bf \L}$^{\Delta}_{n+1}$. There, we have that 
$\frac{n-1}{n} \rightarrow_{n-1} 0 = \frac{n-1}{n}$ and so $(\frac{n-1}{n} \rightarrow_{n-1} 0)\vee \frac{n-1}{n} \not=1$. That is, the equation $(x\rightarrow_{n-1} y)\vee x \approx 1$ does not hold in the class $\lukres^{\Delta}_{n+1}$. By Theorem \ref{wat}, we have that $((\alpha\rightarrow_{n-1} \beta)\iml\alpha)\iml\alpha$ (axiom ({\bf Ax}5$n$) is not a theorem of $\luklog_{n+1}$.
\end{proof}

\

\noindent Therefore, we have the next hierarchy of finite-valued super-{\L}ukasiewicz logics expanded by $\Delta$.

\

$$ \cdots \hspace{.6cm} \luklog_{n+1} \, \subset \, \luklog_{n} \, \subset \, \luklog_{n-1} \hspace{.6cm}  \cdots  \hspace{.6cm}  \subset \, \luklog_{3} \, \subset \,   \luklog_{2}$$

\

\subsection{Strong version of the Completeness Theorem}\label{sectarsk}

We shall prove the strong completeness theorem for the systems $\luklog_n$ using a general technique developed in \cite{FOS1}.
Recall that  a logic  defined over a propositional  language ${\cal S}$ is a system $\mathcal{L}=\langle Fm, \vdash_\mathcal{L}\rangle$ where $Fm$ is the set of formulas over ${\cal S}$ and the relation  $\vdash_{\mathcal{L}}$ is a subset of  $\wp (Fm) \times Fm$ where $\wp(A)$ is the set of all subsets of $A$. The logic $\mathcal{L}$ is said to be  {\em Tarskian} if it satisfies the following properties, for every set $\Gamma\cup\Omega\cup\{\varphi,\beta\}$ of formulas:
\begin{itemize}	
  \item [\rm (1)] if $\alpha\in\Gamma$, then $\Gamma\vdash_{\mathcal{L}}\alpha$, \hfill{(reflexivity)}
  \item [\rm (2)] if $\Gamma\vdash_{\mathcal{L}}\alpha$ and $\Gamma\subseteq\Omega$, then $\Omega\vdash_{\mathcal{L}}\alpha$, \hfill{(monotonicity)}
  \item [\rm (3)] if $\Omega\vdash_{\mathcal{L}}\alpha$ and $\Gamma\vdash_{\mathcal{L}}\beta$ for every $\beta\in\Omega$, then $\Gamma\vdash_{\mathcal{L}}\alpha$.\hfill{(cut)}
\end{itemize}

\noindent A logic $\mathcal{L}$ is said to be {\em finitary}  (or {\em compact}) if it satisfies the following:

\begin{itemize}
  \item [\rm (4)] if $\Gamma\vdash_{\mathcal{L}}\alpha$, then there exists a finite subset $\Gamma_0$ of $\Gamma$ such that $\Gamma_0\vdash_{\mathcal{L}}\alpha$.\hfill {(compactness)}
\end{itemize}

 \noindent  The following condition is to add the {\em structurality} to a Tarskian logic:

\begin{itemize}	
  \item [\rm (5)]  if $\Gamma\vdash_{\mathcal{L}}\alpha$, then $\sigma[\Gamma] \vdash_{\mathcal{L}} \sigma(\alpha)$ for each $\mathcal{L}$-substitution $\sigma$; \hfill{(structurality)}
\end{itemize}

\noindent in this way, we obtain what is known as {\em deductive system}.
\begin{defi} \label{maxi} Let $\mathcal{L}$ be a Tarskian logic and let $\Gamma $ be a set of formulas. We say that every set of formulas is a theory. Moreover,  $\Gamma$ is said to be a consistent theory if there is a formula $\varphi$ such that $\Gamma\not\vdash_{\mathcal{L}}\varphi$. Furthermore, we say that $\Gamma$ is a maximal consistent theory if  $\Gamma,\psi\vdash_{\mathcal{L}}\varphi$ for any formula $\psi\notin\Gamma$; and, in this case, we say $\Gamma$ is maximal respect to $\varphi$.
\end{defi}

A set of formulas $\Gamma$ is closed in $\mathcal{L}$ if the following property holds for every formula $\varphi$: $\Gamma\vdash_{\mathcal{L}}\varphi$ if and only if $\varphi\in\Gamma$. It is easy to see that any maximal consistent theory is a closed one.

\begin{lem}[Lindenbaum-\L o\'s] \label{exmaxnotr} Let $\mathcal{L}$ be a Tarskian and finitary logic. Let $\Gamma\cup\{\varphi\}$ be a set of formulas  such that $\Gamma\not\vdash_{\mathcal{L}}\varphi$. Then, there exists a set of formulas $\Omega$ such that $\Gamma\subseteq\Omega$ with $\Omega$ being maximal consistent theory with respect to the formula $\varphi$ in $\mathcal{L}$.
\end{lem}
\begin{proof}
The proof can be found in \cite[Theorem 2.22]{W}.
\end{proof}
 
 \

Let $\Gamma\cup\{\alpha\}$ be a set of formulas.  We say that $\alpha$ is $\luklog_n$-consequence of $\Gamma$ (in the system $\luklog_n$) and write $\Gamma\vdash_{\luklog_n}\alpha$  if for every valuation $v$ it holds:

$$ v(\gamma)=1, \mbox{ for every } \gamma\in\Gamma, \, \mbox{ implies } \, v(\alpha)=1.$$

\noindent Now, for a given maximal theory $\Gamma$ with respect to $\varphi$, we denote by $\Gamma/\equiv$ the set $\{\overline{\alpha}:\alpha\in\Gamma\}$. It is clear that $\Gamma/\equiv$ is a subset of the $\lukres^{\Delta}_n$-algebra  $\mathfrak{Fm}/\equiv$. Then,

\begin{lem}\label{aux}  Let $\Gamma\cup\{\varphi\}\subseteq\mathfrak{Fm}$, with $\Gamma$ being a non-trivial maximal respect to $\varphi$ in $\luklog_n$. Then:

\begin{itemize}
\item[\rm (i)] if $\alpha\in\Gamma$ and $\overline{\alpha}=\overline{\beta}$, then $\beta\in\Gamma$;

\item[\rm (ii)] $\Gamma/\equiv$ is an  implicative filter tied to $\overline{\varphi}$ of $\mathfrak{Fm}/\equiv$.
\end{itemize}
 \end{lem}
\begin{dem} See \cite[Theorem 4.6.]{FOS1}.
\end{dem}

\begin{theo}\label{morf} Let $\Gamma\cup\{\varphi\}\subseteq\mathfrak{Fm}$, with $\Gamma$ non-trivial maximal respect to $\varphi$ in $\luklog_n$. Then, there is a  map $v:\mathfrak{Fm} \to  \text{{\bf \L}}^{\Delta}_{n}$ which is an  homomorphism 
\end{theo}

\begin{dem} Firstly, let us  consider $A:= \mathfrak{Fm}/\equiv$ and $M:=\Gamma/\equiv$. So, it clear that $A/M$ is a simple  $\mbox{\lukres}^{\Delta}_n$-algebra in virtue of Lemma \ref{aux} and Corollary \ref{maximal}. It is not hard to see that $A$ is isomorphic to $\mathfrak{Fm}/\Gamma$. Therefore, there is $\pi: \mathfrak{Fm} \fun \mathfrak{Fm}/\Gamma$  the canonical homomorphism, where $\mathfrak{Fm}/\Gamma$ is the quotient set determined by the congruence $\equiv_{\Gamma}$, which is defined by  $$\alpha\equiv_\Gamma\beta \, \mbox{ \, iff \, } \Gamma \vdash_{\luklog_n} \alpha \imn \beta \mbox{ and }\, \Gamma\vdash_{\luklog_n} \beta \imn \alpha.$$    
\end{dem}

\

\begin{theo}\label{compP} Let $\Gamma\cup\{\varphi\}\subseteq\mathfrak{Fm}$, $\Gamma\vdash\varphi$ if and only if $\Gamma\vDash\varphi$.
\end{theo}

\begin{dem} The proof immediately follows  from Theorem \ref{WAT} and Theorem \ref{morf}.
\end{dem}

\

 \subsection{First-order version of $\luklog_n$ with equality: the logic $\forall\luklog_n$ }

The first-order logic of $\forall$$\luklog_n$   will be introduced in this section.  Let  $\Theta$  be the first-order signature $\langle\P,\F,\C \rangle$, where $\P$ denotes a non-empty set of predicate symbols, $\F$ is a set of function symbols and $\C$ denotes a set  of individual constants. Let $\mathscr{L}(\Theta)$ be the first-order language induced by $\Theta$ containing the propositional signature of  $\luklog_n$, as well as two quantifier symbols $\forall$  and $\exists$, together with  the punctuation marks, commas and parentheses.

Now, consider a denumerable set  $Var$ of individual variables. The notions of bound and free variables, closed terms, sentences, and substitutability are the usual. We denote by $\mathfrak{Fm}_\Sigma$ the set  of all well-formed formulas (wff for short) and denote by $Ter$ the absolutely free algebra of the terms.  By $\varphi(x/t)$, we denote the formula that results from $\varphi$ by replacing simultaneously all the free occurrences of the variable $x$ by the term $t$.



\noindent A $\Theta$-structure $\mathfrak{A}$ is a pair $\langle {\bf A},\S \rangle$ where $\bf A$ is a complete $\lukres^{\Delta}_n$-algebra (that is, we are asking for all subsets of $A$ to have both a supremum and an infimum),  and the structure: $$\S=\langle S,\{P_{\S}\}_{P\in\P},\{f_{\S}\}_{f\in\F}, \C , (\cdot)^\mathfrak{A} \rangle,$$ where  $S$ is a non-empty set (domain) and $(\cdot)^\mathfrak{A}$ is an interpretation map which assigns:

\begin{enumerate}

\item[$\bullet$]  to each individual constant $c\in \C$, an element $c^\mathfrak{A}$ of $S$; 

\item[$\bullet$] to each predicate symbol $P$ of arity $n$, a function $P^\mathfrak{A}:S^n\to A$; 

\item[$\bullet$] to each functional symbol $f$, a function $f^\mathfrak{A}:S^n\to S$.

\end{enumerate}

Let  $\forall\luklog_n$ the first order Hilbert-style calculus defined (over the language $\mathscr{L}(\Theta)$) by extending the calculus $\luklog_n$ by adding the following:\\[3mm]
{\bf Axioms Schemata} 

\begin{equation}\tag{{\bf Ax}9}
\varphi(x/t)\rightarrowtail\exists x\varphi, \,  \mbox{ if $t$ is a term free for $x$ in $\varphi$,}
\end{equation}
\begin{equation}\tag{{\bf Ax}10}
\forall x\varphi\rightarrowtail\varphi(x/t), \, \mbox{ if $t$ is a term free for $x$ in $\varphi$.}
\end{equation}

\noindent  {\bf Axiom Schema}
\begin{equation}\tag{{\bf Ax}11}
   x\approx x,
 \end{equation}

\noindent {\bf Inference rules}

\begin{equation}\tag{$\exists$-{\bf In}}
 \dfrac{\alpha\to\beta}{\exists x\alpha\to\beta}, \, \mbox{ and $x$ does not occur free in $\beta$,}
  \end{equation}
  \begin{equation}\tag{$\forall$-{\bf In}}
\dfrac{\alpha\to\beta}{\alpha\to\forall x\beta}, \, \mbox{ and $x$ does not occur free in $\alpha$.}
  \end{equation}
  \begin{equation}
  \tag{R-$\approx$ } \dfrac{x\approx y }{ \varphi\to\varphi(x\wr y)}, \, \mbox{ we call this rule as  Leibniz law } 
    \end{equation}
  
   \begin{equation}
   \mbox{where  $y$ is a free variable for $x$ in $\varphi$, and $\varphi(x\wr y)$  denotes any formula obtained from  $\varphi$ replacing some, but not necessarily all, free occurrences of $x$ by $y$.}
  \end{equation} 
 
 \
 \begin{prop}{\rm \cite{FOS1}} For a term $t_1$, and $x,y$ and $z$ individual variables, we have:
\begin{multicols}{2}
\begin{itemize}
  \item[\rm (i)] $\vdash \forall x ( x\approx x)$,
  \item[\rm (ii)] $\vdash t_1\approx t_1$,
  \item[\rm (iii)] $\{ x\approx y  \}\vdash y \approx x $,
  \item[\rm (iv)] $\{ x\approx y, y\approx z \}\vdash x\approx z$,
\end{itemize}
\end{multicols}
\end{prop}

\noindent  As above, we denote by $\vdash_{\forall\luklog_n}\alpha$ to indicate that there exists a derivation of $\alpha$ in  $\forall\luklog_n$, and by $\Gamma\vdash_{\forall\luklog_n}\alpha$ if there is a derivation of $\alpha$ from $\Gamma$. Besides, we denote by $\vdash_{\forall\luklog_n}\varphi\leftrightarrow\psi$ to indicate that  $\vdash_{\forall\luklog_n}\varphi\to\psi$ and $\vdash_{\forall\luklog_n}\psi\to\varphi$. \\[2mm]
A $\mathfrak{A}$-assignment is a map $v:Var\to S$. By $v[x\to a]$ we denote the the following $\mathfrak{A}$-assignment:  $v[x\to a](x)=a$ and $v[x\to a](y)=v(y)$ for any $y\in V$ such that $y\neq x$ and any $a\in S$. \\
Let $\mathfrak{A}=\langle {\bf A},\S\rangle$ be a $\Theta$-structure and $v$ a $\mathfrak{A}$-assignment. We define the values of the terms and the truth values of any wff in $\mathfrak{A}$ for the assignment $v$ as follows:

\begin{center}
$||c||^\mathfrak{A}_v=c^\mathfrak{A}$ if $c\in S$,\\ [2mm]

$||x||^\mathfrak{A}_v=v(x)$ if $x\in Var$,\\ [2mm]

$||f(t_1,\cdots ,t_n)||^\mathfrak{A}_v=f^\mathfrak{A}(||t_1||^\mathfrak{A}_v,\cdots ,||t_n||^\mathfrak{A}_v)$, for any $f\in\F$,\\ [2mm]

$||P(t_1,\cdots ,t_n)||^\mathfrak{A}_v=P^\mathfrak{A}(||t_1||^\mathfrak{A}_v,\cdots ,||t_n||^\mathfrak{A}_v)$, for any $P\in\P$,\\ [2mm]
$||\alpha\rightarrowtail\beta||^\mathfrak{A}_v=||\alpha||^\mathfrak{A}_v\rightarrowtail ||\beta||^\mathfrak{A}_v$,\\ [2mm]

$||\Delta\alpha||^\mathfrak{A}_v=\Delta ||\alpha||^\mathfrak{A}_v$,\\ [2mm]
{\red }
$||\forall x\alpha||^\mathfrak{A}_v=\underset{a\in S}{\bigwedge} ||\alpha||^\mathfrak{A}_{v[x\to a]}$,\\ [2mm]
$||\exists x\alpha||^\mathfrak{A}_v=\underset{a\in S}{\bigvee}||\alpha||^\mathfrak{A}_{v[x\to a]}$.

\end{center}

\noindent Now, it is easy to see that the following property   $||\varphi(x/t)||^\mathfrak{A}_v= ||\varphi||^\mathfrak{A}_{v[x\to ||t||^\mathfrak{A}_v]}$ holds. 
We say that $\mathfrak{A}$ and $v$ {\bf satisfy} a formula $\varphi$, denoted by $\mathfrak{A}\vDash\varphi[v]$, if $||\varphi||^\mathfrak{A}_v=1$. Besides, we say that $\varphi$ is {\bf valid} in $\mathfrak{A}$ if $||\varphi||^\mathfrak{A}_v=1$ for each  $\mathfrak{A}$-assignment $v$, and we denote $\mathfrak{A}\vDash\varphi$. We say that $\varphi$ is a {\bf semantical consequence} of $\Gamma$ in $\forall\luklog_n$, if for any structure $\mathfrak{A}$ it holds:
 $$\mathfrak{A}\vDash\gamma \mbox{  for each } \gamma\in\Gamma, \mbox{ implies } \mathfrak{A}\vDash\varphi.$$ 
 In this case, we denote it by $\Gamma\vDash\varphi$.

 We define the values of terms and the truth values of the formulas in $\mathfrak{A}$ for a valuation $v$ extending the Definition above adding the following condition: $||t_1\approx t_2||_v^\mathfrak{A}=1$ if and only if $||t_1||_v^\mathfrak{A}=||t_2||_v^\mathfrak{A}$. For a given set of formulas $\Gamma\cup \{\alpha\}$, the semantical consequence of $\alpha$ from $\Gamma$ that we denote by $\Gamma\vDash\alpha$  is defined as usual.

\

\begin{theorem} {\rm (Soundness Theorem).}\label{compleprimererord} Let $\Gamma\cup\{\varphi\}\subseteq\mathfrak{Fm}_{\Sigma}$, if $\Gamma\vdash_{\forall\luklog_n}\varphi$ then $\Gamma\vDash\varphi$.
\end{theorem}
\begin{proof} 
 From the fact that each \lukres$^{\Delta}_n$-algebra is in fact a  Monteiro algebra in term of \cite[Section 6]{FOS1}. The proof of this Theorem is a particular case that \cite[Theorem 5.6.]{FOS1}.
\end{proof}

\ 

It is possible to define the notion  that a theory $\Gamma$ is maximal respect to the some formula $\varphi$, see Definition \ref{maxi} for the logic  $\Gamma\vdash_{\forall\luklog_n}$. Moreover, the Lindenbaum-\L o\'s Lemma holds for  $\Gamma\vdash_{\forall\luklog_n}$. From the latter considerations and the fact that each \lukres$^{\Delta}_n$-algebra is in fact a  Monteiro algebra we have proved the following theorem.

 \begin{prop}{\rm \cite{FOS1}} For given  terms $t_1,t'_1,\cdots ,t_n,t'_n$, we have:
\begin{itemize}
  \item[\rm (i)] $\{ t_1 \approx t_2  \}\vdash t_2 \approx t_1$,
  \item[\rm (ii)] $\{ t_1\approx t_2, t_2\approx t_3 \}\vdash t_1\approx t_3$,
  \item[\rm (iii)] $\{ t_1\approx t'_1, t_2\approx t'_2, t_3\approx t'_3,\cdots,t_{n-1}\approx t'_{n-1}\}\vdash f(t_1,\cdots,t_n)\approx f(t'_1,\cdots,t'_n)$, for any function symbol $f$ of arity $n$,
  \item[\rm (iv)] $ \{ t_1\approx t'_1, t_2\approx t'_2, t_3\approx t'_3, \cdots, t_{n}\approx t'_{n} \}\vdash \varphi(\overrightarrow{x}/\overrightarrow{t})\to\varphi(\overrightarrow{x}/\overrightarrow{t'})$, for any formula without quantifiers $\varphi$ depending at most on the variables $x_1,\cdots,x_n$. Instead of $\xi_1,\cdots,\xi_n$ (where $\xi_i$'s are terms or formulae and $n$ is arbitrary or fixed by the context), we write just $\overrightarrow{\xi}$.
\end{itemize}
\end{prop}
\begin{proof} 
It follows immediately from the very definitions.
\end{proof}

\begin{theorem} {\rm (Completeness of $\forall\luklog_n$ w.r.t. the class of  \lukres$^{\Delta}_n$-algebras).}\label{compl} Let $\Gamma\cup\{\varphi\}$ be a set of formulas.  Then: $\Gamma\vDash\varphi$ implies that $\Gamma\vdash\varphi$.
\end{theorem}
\begin{proof} 
It follows the corresponding proofs of \cite[Section 6]{FOS1} as particular case.
 \end{proof}

 It is worth mentioning that in Cintula and Noguera's paper \cite{CN2}, it was presented a generic Completeness Theorem for certain algebraizable logics. So, Theorem \ref{compl} can be obtained using their method. However, the proof given in Figallo-Orellano and Slagter's paper \cite{FOS1} is  different since there it is used a different notion of maximal consistent theory and the {\em canonical model} is built over set of formulas instead of sentences.

\section{Infinite-valued $\Delta$-\L ukasiewicz residuation algebras with bottom }

In this section, we introduce a new class of algebras, more general that the ones considered in Section \ref{s3}, and that properly contains all the classes considered there.

\begin{definition}\label{def41} A \L ukasiewicz residuation algebra with $\Delta$ (or a $\Delta$-\L ukasiewicz residuation algebra) is an algebra $\A=\langle A, \iml, \Delta, 1\rangle$ such that it satisfies identities (\L1)--(\L5) and
\begin{enumerate}[\rm ($\Delta$\L R$1$)] 
\item $\Delta x\iml x\approx1$,
\item $\Delta x \iml y \approx \Delta x \iml(\Delta x\iml y)$, 
\item If $z\iml y \approx z\iml (z\iml y)$   and  $z\preceq x$, then $z\preceq \Delta x$, 
\item $\Delta(x\iml z) \iml (\Delta x \iml \Delta z) \approx 1$.
\end{enumerate}
We denote by $\lukres^{\Delta}$ the class of $\Delta$-\L ukasiewicz residuation algebras.
\end{definition}
\

It is worth mentioning that the class of $\Delta$-\L ukasiewicz residuation algebras constitute a quasi-variety. The intuition behind quasi-identity ($\Delta$\L R$3$) is to assure that $\Delta x$ is the greatest Tarskian element below $x$. We  follow  ideas presented in \cite{EFFG}.

\begin{lemma} Let  $\A$ be an $\lukres^{\Delta}$-algebra. Then, for every $x,y\in A$ the following conditions are satisfied:
 \begin{multicols}{2}
\begin{enumerate}
    \item[\rm ($\Delta$\L R$5$)] $\Delta 1=1$,
    \item[\rm ($\Delta$\L R$6$)] If $x\leq y$, then $\Delta x\leq \Delta y$,
    \item[\rm ($\Delta$\L R$7$)] $\Delta\Delta x = \Delta x$
    \item[\rm ($\Delta$\L R$8$)] $\Delta(\Delta x\rightarrowtail \Delta y)=\Delta x\rightarrowtail \Delta y$
    \item[\rm ($\Delta$\L R$9$)] $T(\A)=\Delta(\A)$,
  
    \item[\rm ($\Delta$\L R$10$)] $\Delta x\rightarrowtail (y\rightarrowtail  z)=(\Delta x\rightarrowtail y)\rightarrowtail(\Delta x\rightarrowtail z)$, 
 
    \item[\rm ($\Delta$\L R$11$)] $\Delta x\rightarrowtail\Delta(x\rightarrowtail\Delta x)=1$,
    \item[\rm ($\Delta$\L R$12$)] $\Delta x\rightarrowtail\Delta(x\rightarrowtail y)=\Delta x\rightarrowtail\Delta y$.
  
\end{enumerate}
\end{multicols}
\end{lemma}
\begin{proof} They are consequence of Definition \ref{def41}.
\end{proof}

\

\noindent The notion of (maximal) implicative filter for a given  $\lukres^{\Delta}$-algebra $\A$ is the same of Section \ref{imfiltersect}. Besides, if $\A$ is a \L ukasiewicz residuation algebra and $D\subseteq A$ an implicative filter, then it is easy to see that the relation
$$R(D)=\{(x,y)\in A^2 : x\iml y, y\iml x\in D\}$$
is an equivalence relation on $\A$ compatible with \, $\iml$ . Besides, the quotient structure \break $\A/D=\langle A/R(D), \iml, \bar{1}\rangle$  defined in the usual way is also a \L ukasiewicz residuation algebra. 

\begin{proposition} Let $\A$ be a \L ukasiewicz residuation algebra and $D\subseteq A$ an implicative filter. Then, if $z\in T(\A)$ then $\bar{z}\in T(\A/D)$. That is, if $z$ is Tarskian in $\A$, then the equivalence class $\bar{z}$ is Tarskian in $\A/D$.
\end{proposition}
\begin{proof} Immediate.
\end{proof}

\

\noindent Next, we introduce the notion of {\em $\Delta$-filter}.

\begin{definition}\label{deltafilter} Let $\A\in \lukres^{\Delta}$. We say that the implicative filter $D\subseteq A$ is a $\Delta$-filter if it satisfies the following conditions: for all $x,y,z\in A$
\begin{enumerate}[\rm (i)]
\item if $x\in D$ then $\Delta x\in D$,
\item if $(z\iml(z\iml y))\iml(z\iml y)\in D$ and $z\iml x\in D$, then $z\iml \Delta x\in D$,
\end{enumerate}
We denote by $\D_{\lukres^{\Delta}}(\A)$ the collection of all $\Delta$-filters of $\A$.
\end{definition}

\

\noindent Then,

\begin{lemma} Let $\A$ be  a  $\Delta$-{\L}ukasiewicz residuation algebra; and $D\in \D^{\Delta}_{\infty}(\A)$. Then, the quotient structure $\A/D=\langle A/R(D), \iml, \Delta, 1\rangle$ is a member of the class $\lukres^{\Delta}$.
\end{lemma}
\begin{proof} We know that $R(D)$ is an equivalence relation compatible with \, $\iml$ . From condition (i) of Definition \ref{deltafilter} and ($\Delta$\L R$4$), we have that $R(D)$ is also compatible with $\Delta$.  Besides, it is routine to check that  $\A/D$ verifies the condition ($\Delta$\L R$1$), ($\Delta$\L R$2$) and ($\Delta$\L R$4$). Let us see that it also verifies  ($\Delta$\L R$3$). Indeed, let $x, y, z\in A$ such that (1) $\bar{z}\iml \bar{y} = \bar{z}\iml (\bar{z}\iml \bar{y})$  and (2)  $\bar{z}\leq \bar{x}$. From (1), we have that  $(z\iml(z\iml y))\iml(z\iml y)\in D$; and from (2), we have  $\bar{z}\iml \bar{x} = \bar{1}= D$ and then (4) $z\in x\in D$. Hence, from (3), (4) and condition (ii) of  Definition \ref{deltafilter}, we have that $z\iml \Delta x\in D$, that is, $\bar{z}\iml \Delta\bar{x} = \bar{1}$. 
\end{proof}

\

\begin{proposition} \label{proptars} Let $\A, \B\in \lukres$ and let $h:\A\fun \B$ a $\lukres$-homomorphism. If $t\in T(\A)$, then $h(t)\in T(h(\A))$.
\end{proposition}
\begin{dem} Let $b\in B$ such that there is $a\in A$ such that $h(a)=b$. Then, $h(t)\iml b = h(t) \iml h(a) =h(t\iml a) = h(t\iml(t\iml a))= h(t)\iml(h(t)\iml h(a))=h(t)\iml(h(t)\iml b)$. That is, $h(t)$ is a Tarskian element of $h(\A)$.
\end{dem}

\

\begin{proposition} Let $\langle A, \iml, 1\rangle$ be a \L ukasiewicz residuation algebra and let $\Delta_1$ and $\Delta_2$ be two unary operators on $A$ such that both satisfy  ($\Delta$\L R$1$)--($\Delta$\L R$4$). Then, \, $\Delta_1=\Delta_2$. That is, every \L ukasiewicz residuation algebra admits at most one structure of $\lukres^{\Delta}$-algebra.
\end{proposition}
\begin{dem} Let $x\in A$. Then, we have that  (1) $\Delta_1 x\leq x$ \, and \, (2) $\Delta_1x\iml y = \Delta_1x\iml(\Delta_1x\iml y)$, for every y. On the other hand, from ($\Delta$\L R$3$) for $\Delta_2$ we have that, (3) for every $y, z\in A$, if $z\iml y=z\iml(z\iml y)$ and $z\leq x$, then $z\leq \Delta_2 x$. From (1), (2) and (3), $\Delta_1 x\leq \Delta_2 x$. Analogously we have that $\Delta_2 x\leq \Delta_1 x$.
\end{dem}

\
  
\begin{prop} Let $\langle A, \iml, 1\rangle$ be a \L ukasiewicz residuation algebra and let $\Delta$ be a unary operators on $A$ such that both satisfy  ($\Delta$\L R$1$)--($\Delta$\L R$4$). Then, \, $\Delta x = \max \, T_x$ where \break $T_x=\{z\in T({\bf A}): z\leq x\}$. 
\end{prop}
\begin{dem} Immediate.
\end{dem}

\

\begin{prop}\label{central}
Let $\bf A$ and $\bf A'$  be  $\lukres^{\Delta}$-algebras and let $h:A\to A'$ be a \L ukasiewicz residuation isomorphism. Then, $h(T_x)=T_{h(x)}$; moreover, $h(\Delta x)=\Delta h(x)$.
\end{prop}
\begin{proof}
Assume that $z\in h(T_x)$, then there  is $w\in T_x$ such that $h(w)=z$. So, $w\leq x$ and therefore $h(w)\leq h(x)$. Taking into account Proposition \ref{proptars}, we known that $h(w)$ is a Tarskian element. Thus, $z=h(w)\in T_{h(x)}$ as desired. Conversely, let's $y\in T_{h(x)}$ then there is only element $z\in A$ such that $h(z)=y\leq h(x)$ and $z$ is also Tarskian element.  Hence, $h(z\to x)=1$ and taking into account that $h$ is a one-to-one homomorphism, we have that $z\leq x$. Therefore, $y=h(z)\in h(T_x)$.\\
On the other hand, it is clear now that $h(z)\leq  h(\max T_x)$ and so $\max \{h(z) : h(z)\in h(T_x)\}\leq  h(\max \, T_x)$. Since, $h(T_x)=T_{h(x)}$ and the fact that $h$ is one-to-one we have $\Delta h(x)=\max \, T_{h(x)} \leq  h( \max  \, T_x)=h(\Delta x)$. Taking into account $\Delta x=\max T_x$ is a Tarskian element (by ($\Delta\L$R$2$)) such that $\Delta x\leq x$. Hence,  $h(\Delta x)=h(\max \, T_x)\leq h(x)$. In virtue of Proposition \ref{proptars} we have  that $h(\max \, T_x)$ is a Tarskian element, we have therefore $h(\Delta x)=h(\max \,  T_x)\leq \max \, T_{h(x)}=\Delta h(x)$, which completes the proof.
\end{proof}

\

Consider the $\lukres$-algebra $\mbox{\bf \L}_{\infty}=\langle [0,1], \iml, 1\rangle$ of Facts \ref{exam1} (ii). It is easy to check that $T(\mbox{\bf \L}_{\infty})=\{0,1\}$ and if we define $\Delta : [0,1]\fun[0,1]$ as in Example \ref{exampleLn}, then  $\mbox{\bf \L}^{\Delta}_{\infty}=\langle [0,1], \iml, \Delta, 1\rangle$  is a $\lukres^{\Delta}$-algebra. 

\begin{proposition}\label{propsub} Let $S$ be a $\lukres$-subalgebra of $\mbox{\bf \L}_{\infty}$ and let $m=\min S$. Then,  $T(S)\subseteq \{m,1\}$.
\end{proposition}
\begin{dem} Let $z\in S$ a Tarskian element. Then, $z\iml y = z\iml(z\iml y)$, for every $y\in S$. Let $y\leq z$  then $z\iml y=\min\{1,1-z+y\} = 1-z+y$. Besides, $z\iml(z\iml y)=\min\{1, 1-z + (1-z+y)\}=\min\{1, 2-2z+y\}$. If $z\iml(z\iml y)=1$, then $z=y$; and, if $z\iml(z\iml y)=2-2z+y$, $z=1$.\\
Therefore, if $m\in S$ we have that $T(S)= \{m,1\}$ and if $m\notin S$, $T(S)= \{1\}$.
\end{dem}

\

\begin{remark}\label{remsub} Let $\A$ be a non-trivial $\lukres^{\Delta}_{\infty}$-algebra and suppose that there exists an $\lukres$-monomorphism  $h:\A\iml \mbox{\bf \L}_{\infty}$. Then, $h(\A)$ is a $\lukres$-subalgebra of $\mbox{\bf \L}_{\infty}$ which has first element. Indeed, since $\A$ is non-trivial, there is $x\in A$ such that $x<1$ and therefore $\Delta x \not=1$. By Proposition \ref{proptars}, $h(\Delta x)$ is a Tarskian element of $h(\A)$ different from $1$, since $h$ is injective.  By Proposition \ref{propsub}, $h(\Delta x)$ is the first element of $h(\A)$. 
\end{remark}

\


Now, we shall consider $\lukres^{\Delta}$-algebras with first element. More precisely, we call $\Delta$-\L ukasiewicz residuation algebra with first element (or $\lukres^{\Delta}_0$-algebra) to any structure $\A=\langle A, \iml, \Delta, 0, 1\rangle$ of type $(2,1,0,0)$ such that (i) the reduct $\langle A, \iml, \Delta, 1\rangle$ is an $\lukres^{\Delta}$-algebra, and (ii) it verifies the equation 
$$0\iml x \approx 1$$ 

\

\noindent It is clear that $\mbox{\bf \L}^{\Delta}_{\infty}=\langle [0,1], \iml, \Delta, 0, 1\rangle$, where $\iml$ is defined as in Facts \ref{exam1} (ii) and $\Delta$ is defined as in Example \ref{exampleLn}, is a $\Delta$-\L ukasiewicz residuation algebra with first element. Then,

\begin{lemma} Let $\A$ be a non-trivial $\lukres^{\Delta}_{0}$-algebra. Then, there exists $X$ and a  map $h:\A\fun \big(\mbox{\bf \L}^{\Delta}_{\infty}\big)^{X}$ such that $h$ is an inmersion, i.e., $h$ is a  $\lukres^{\Delta}_{0}$-monomorphism.
\end{lemma}
\begin{dem} It follows immediately from  the well-known representation theorem for MV-algebras (orWajsberg algebras, see for instance \cite{CDM}) and Proposition \ref{central}.
\end{dem}

\

An immediate consequence of the last theorem is the following corollary.
 
\begin{corollary}{\rm (\cite[Lemma 4]{EGM2001})}\label{central} The quasi-variety of $\lukres^{\Delta}_{0}$-algebras is semi-simple one and the simple algebras of the varieties are $\mbox{\bf \L}^{\Delta}_{\infty}=\langle [0,1], \iml, \Delta, 0, 1\rangle$  and their sub-algebras.
\end{corollary}
 
As a by-product of the last Corollary, it is easy to see that the class of $\Delta$-\L ukasiewicz residuation algebras with first element is term-equivalent to the variety of MV$_\Delta$-algebras introduced by H\'ajek, see for instance \cite[pag. 45]{EGM2001}.


\subsection{A $\Delta$-fuzzy logic}  
In this section, we present a logic, by means of a Hilbert systems, whose algebraic counterpart are precisely the 
$\Delta$-\L ukasiewicz residuation algebras with first element. This logic turns out to be a $\Delta$-fuzzy logic  which is an alternative axiomatization  of \L ukasiewiz logic \L$_\Delta$ introduced by H\'ajek in his celebrated book. Let us consider a denumerable set $Var$ of propositional variables, and let $Fm$ be the propositional language generated by $Var$ over the signature $\{\iml, \Delta, \bot, \top\}$, and $\mathfrak{Fm}$ the absolutely free algebra with this signature.

\

\noindent The Hilbert-style calculus $\luklog_{\bot}$ is obtained by axioms ({\bf Ax}1)--({\bf Ax}4) along with the following: \\[2mm]
\begin{equation}\tag{{\bf Ax}9}
\bot \iml \alpha
   \end{equation}
\vspace{-0.5cm}
\begin{equation}\tag{{\bf Ax}10}
\Delta \alpha\iml \alpha
   \end{equation}
   \vspace{-0.5cm}
\begin{equation}\tag{{\bf Ax}11}
(\Delta\alpha \iml \beta)\iml (\Delta \alpha \iml (\Delta \alpha\iml \beta)) 
    \end{equation}
   \vspace{-0.5cm}
\begin{equation}\tag{{\bf Ax}12}
(\Delta \alpha \to (\Delta \alpha\iml \beta)) \to (\Delta\alpha \to \beta) 
    \end{equation}
    \vspace{-0.5cm}
    \begin{equation}\tag{{\bf Ax}13}
     \Delta(\alpha \iml \beta) \iml (\Delta \alpha \iml \Delta \beta)
    \end{equation}

\noindent and the rules are {\em modus ponens} along with 
 \begin{equation}\tag{QGEN}
 \displaystyle\frac{(\gamma \iml \beta) \iml (\gamma \iml (\gamma\iml \beta)), (\gamma\iml \beta) \iml (\gamma \iml (\gamma \iml \beta)),\gamma\iml \alpha}{\gamma \iml\Delta \alpha} 
\end{equation}

\noindent We also  consider  non-primitive connectives $\vee$, $\sim$ and $\wedge$ defined as follows: $$\alpha \vee \beta:= (\alpha \rightarrow \beta)\rightarrow \beta,$$
 $$\sim \alpha:=\alpha \iml \bot,$$
 $$ \alpha\wedge \beta:= \sim(\sim \alpha \vee \sim \beta).$$

It is clear that the logic $\luklog_{\bot}$ is an alternative presentation to the one given in  \cite[pag. 45]{EGM2001} for \L ukasiewicz logic with $\Delta$ in virtue of Corollary \ref{central}. So, we can present Adequacy Theorems both for the propositional level and the quantified version in a same way to the one's given by H\'ajek in his book, \cite{PH}.

\section{Conclusions}

In the present paper we have studied Bazz's $\Delta$ operator in the context of $\{\iml,\top \}$-fragment of \L ukasiewicz logic.  In first palce, we study the class of $n$-valued \L ukasiewicz residuation algebras expanded with $\Delta$. We prove important properties of these algebras and calculate the cardinality of the free algebra with a finite number of free generators. Then, we propose a family of $n$-valued logics for which the $n$-valued \L ukasiewicz residuation algebras expanded with $\Delta$ are an algebraic counterpart. Besides, we propose a suitable fisrt-order version of these logics and prove the corresponding soundness and completeness theorems. Finally, we present the infinite-valued $\Delta$-\L ukasiewicz residuation algebras with bottom which turns out to be an alternative presentation   for \L ukasiewicz logic with $\Delta$ given in  \cite{EGM2001}, for instance.

It remains open the study of infinite-valued $\Delta$-\L ukasiewicz residuation algebras without considering a first element.



\bibliographystyle{plain}

\end{document}